\documentclass[final,4p,4pt, times]{article}
\usepackage{authblk}
\usepackage{amsmath}
\usepackage{amssymb,latexsym}
\usepackage[german,english]{babel}
\usepackage{url}
\usepackage{graphicx}
\usepackage{gastex}
\usepackage{longtable}
\usepackage{lscape}
\usepackage{tabularx}
\usepackage{multicol}
\usepackage{verbatim}
\usepackage{multirow}
\usepackage{graphicx}
\usepackage{epsfig, graphics, graphicx}
\usepackage{geometry}
\usepackage{float}
\usepackage{tikz}
\usetikzlibrary{decorations.pathreplacing}
\usepackage{amsmath,blkarray,booktabs,bigstrut}
\usepackage{subcaption,lipsum}
\usepackage{lscape}
\newcommand\undermat[2]{%
	\makebox[0pt][l]{$\smash{\underbrace{\phantom{%
					\begin{matrix}#2\end{matrix}}}_{\text{$#1$}}}$}#2}

\hbadness=\maxdimen

\newtheorem{thm}{{\bf Theorem}}[section]
\newtheorem{prop}{{\bf Proposition}}[section]
\newtheorem{cor}{{\bf Corollary}}[section]

\usepackage{chngcntr}
\counterwithout{equation}{section}
\date{}

\begin{document}
	
	\title{Spectral Properties of Power Graphs of Metacyclic Groups}

\author[1]{Aditya Singh}
\author[2]{Yogendra Singh}
\author[1]{Anand Kumar Tiwari}
        
\affil[1]{\small Department of Applied Science, Indian Institute of Information Technology, Prayagraj 211015, India,}
     
\affil[2]{\small Department of Mathematics, Faculty of Sciences, Adani University, Ahmedabad 382421, India} 
\maketitle
	
	\hrule
\begin{abstract}
For a group $\Omega$, the associated power graph $P(\Omega)$ is defined as the graph whose vertices are the elements of $\Omega$, with two distinct vertices $u,v\in \Omega$ being adjacent if either $u=v^m$ or $v=u^n$ for some $m,n \in \mathbb{N}$. In this paper, we completely characterise the structure of the power graph associated with the class of metacyclic groups. Building on this structural description, we derive explicit expressions for the characteristic polynomials of the adjacency, Laplacian, and signless Laplacian matrices. Moreover, we obtain lower and upper bounds for the spectral radii of the adjacency and signless Laplacian matrices.

\end{abstract}
\smallskip



\hrule

\section{Introduction}
The study of various algebraic structures via associated graphs has been a fascinating area of research over the past few decades; see \cite{c(2022)}. In the present study, we focus on the power graphs of metacyclic groups, which are generalisations of the dihedral group $D_{2n}$, the semidihedral group $SD_{8n}$, and the generalised quaternion group $Q_{4n}$.

Kelarev and Quinn \cite{kq(2002)} introduced directed power graphs for semigroups. For a semigroup $S$, the directed power graph $P(S)$ has vertex set $S$, and a directed edge exists between two distinct vertices $u$ and $v$ whenever $v=u^m$ for some positive integer $m$. Subsequently, Chakrabarty \textit{et al.} \cite{cgs(2009)} studied the undirected version of the power graph of a group, where two different vertices are adjacent whenever one element is a power of the other. They further characterised finite groups whose power graphs are complete by proving that $P(\Omega)$ is complete exactly when $\Omega\cong \mathbb{Z}_n$, with $n=1$ or $n=q^{\ell}$ for some prime $q$ and positive integer $\ell$. In \cite{cg(2011)}, Cameron and Ghosh proved that two finite abelian groups having isomorphic power graphs must themselves be isomorphic. Curtin and Pourgholi \cite{cp(2014)} established that among all finite groups of a fixed order, cyclic groups possess the maximum number of edges and the largest clique size in their power graphs. A survey on power graphs first appeared in \cite{akc(2013)}, followed by a more recent review article in \cite{kscc(2021)}. Further developments and recent contributions in this area can be found in \cite{bcdd(2024),bddg(2026),cjs(2026),stpa(2025)}.

In the existing literature, the spectral study of power graphs has attracted considerable attention due to its applications in areas such as group theory, ring theory, and the theory of Lie algebras \cite{m(2015),prr(2014)}. Over the last decade, considerable attention has been devoted to determining the spectra of algebraic graphs: see \cite{ba(2020), b(2026), cp(2015), gb(2018)}. The authors \cite{ar(2019), mga(2017)} computed the spectra and Laplacian spectra of power graphs of Mathieu, $\mathbb{Z}_n$, and $D_{2n}$. Motivated by these, we determine several spectral properties of power graphs for the well-known split metacyclic groups.

The organisation of the paper is outlined below. Section \ref{Sec2} contains the basic terminology, notation, and preliminary concepts required throughout the paper. In Section~\ref{Sec3}, we investigate the power graphs of a particular family of split metacyclic groups and discuss their structural properties. Section \ref{Sec4} is devoted to the determination of the characteristic polynomials of the adjacency, Laplacian, and signless Laplacian matrices associated with these power graphs by employing equitable partitions and quotient matrices. Explicit computations for the case $k=2$ and odd prime $q$ are also included. In Section \ref{Sec5}, bounds for the spectral radii of the adjacency and signless Laplacian matrices are established, along with a comparative table for $3\leq q\leq 13$ when $k=2$. The paper concludes in Section \ref{s6} with concluding observations and possible future directions.

\section{Preliminaries} \label{Sec2}
This section covers several fundamental graph-theoretic concepts and well-known results that we use throughout the article. We refer to see  \cite{bh(2012)} for more details. Suppose $\Gamma$ is an undirected, simple, connected graph with the vertex set $V(\Gamma)$ and edge set $E(\Gamma)$. Two vertices $v_{1}$ and $v_{2}$ are adjacent if there is an edge between them, symbolised as $v_{1}\sim v_{2}.$ The degree of a vertex $u$, denoted as $\deg(u)$, is the number of vertices adjacent to $u$. The join of graphs $\Gamma_{1}$ and $\Gamma_{2}$, denoted as $\Gamma_{1} + \Gamma_{2}$, is the graph with the vertex set $ V(\Gamma_{1}) \cup V(\Gamma_{2})$ and edge set $ E(\Gamma_{1}) \cup E(\Gamma_{2}) \cup \{ c \sim d : c \in V(\Gamma_{1}),~ d \in V(\Gamma_{2})\}$. Let $n\Gamma$ denote the disjoint union of $n$ copies of $\Gamma$ and $K_{n}$ denote the complete graph on $n$ vertices.

The adjacency matrix of $\Gamma$ is an $n \times n$ matrix of the form $A(\Gamma) = (a_{ij})$, where $a_{ij} = 1$ if $i \sim j$, and $0$ otherwise. The characteristic polynomial of $\Gamma$ is given by $Char(\Gamma, ~x) = \det(xI-A(\Gamma)).$ The eigenvalues of $\Gamma$ are the eigenvalues of $A(\Gamma)$ and are indicated by $\lambda_{t}(\Gamma)$, where $t = 1, 2, \ldots, n$. The matrices $L(\Gamma)=D(\Gamma)-A(\Gamma)$ and $Q(\Gamma)=D(\Gamma)+A(\Gamma)$ are the Laplacian and the signless Laplacian matrices of $\Gamma$, respectively. The Laplacian (signless) eigenvalues of $\Gamma$ are the eigenvalues of  $L(\Gamma)$ $(\textrm{resp.} \, Q(\Gamma))$ and are indicated by  $\mu_{t}(\Gamma)$ $(\textrm{resp.}\, \delta_{t}(\Gamma))$, where $t = 1, 2, \ldots, n$. Since these matrices are symmetric, all of their eigenvalues are real and hence can be arranged as $\lambda_{1}(\Gamma) \geq \lambda_{2}(\Gamma) \geq \cdots \geq \lambda_{n}(\Gamma)$, $\mu_{1}(\Gamma) \geq \mu_{2}(\Gamma) \geq \dots \geq \mu_{n}(\Gamma)$ and $\delta_{1}(\Gamma)\geq \delta_{2}(\Gamma)\geq \dots \geq \delta_{n}(\Gamma)$. Further details about these matrices can be seen in \cite{bh(2012), Panda(2019)}. The spectrum of $\Gamma$ is the collection of all the eigenvalues with their multiplicities. The maximum eigenvalue of $\Gamma$ is called the spectral radius of $\Gamma$. Let $\mathcal{J}_n$ denote the matrix of order $n$ whose each entry is 1, and matrices $O_n$ and $I_n$ as the zero and identity matrix of order $n$, respectively. Also, the notation $A^t$ denotes the transpose of $A$.

Let $\Gamma = \Gamma(V,E)$ be a graph of order $n$ and $\Gamma_i = \Gamma_i(V_i,E_i)$ be a graph of order $n_i$, where $i=1,\ldots,n$. The joined union graph $\Gamma[\Gamma_1,\ldots,\Gamma_n]$ of $\Gamma_1, \Gamma_2, \ldots, \Gamma_n$ with respect to $\Gamma$ is the graph obtained from the union of graphs $\Gamma_1,\Gamma_2,\dots,\Gamma_n$ by joining every vertex of $\Gamma_i$ to every vertex of $\Gamma_j$ whenever $v_i$ and $v_j$ are adjacent in $\Gamma$. The usual join of two graphs $\Gamma_1$ and $\Gamma_2$ is a particular case of joined union graph written as $K_2[\Gamma_1,\Gamma_2]= \Gamma_1 + \Gamma_2$, where $K_2$ is the complete graph of order $2$. A partition $\pi : V_1 \cup V_2\cup \cdots \cup V_m$ of the vertex set $V$ of a graph $\Gamma$ is equitable if for each $i$ and for all $u,v \in V_i,$ $|N(u)\bigcap V_j| =|N(v)\bigcap V_j|$ for all $j.$ For the equitable partition $ \pi$ of the vertex set $V$ of $\Gamma$, we define the quotient matrix $\mathcal{Q} = [q_{ij}]_{m}$, where $q_{ij} = |N(v)\bigcap V_j|$, for $v \in V_i$. 

Now, recall the following results that we use throughout the paper.

The characteristic polynomial of the quotient matrix divides the characteristic polynomial of its parent matrix. Moreover, the theorem below describes how the polynomial corresponding to the quotient matrix is related to that of the matrix obtained from the joined union graph.

\begin{thm}[Schwenk \cite{s(1974)}]\label{t2.1}
Let $\Gamma$ be a connected graph on the vertex set $\{1,2,\ldots,q\}$. Assume that, for each $1\leq i\leq q$, the graph $\Gamma_i$ is $r_i$-regular. Then the family $\{V(\Gamma_1),V(\Gamma_2),\ldots,V(\Gamma_q)\}$ forms an equitable partition of the joined union graph $\Gamma[\Gamma_1,\Gamma_2,\ldots,\Gamma_q].$ Let $\mathcal{Q}$ be the quotient matrix corresponding to this partition. Then the adjacency characteristic polynomial of $\Gamma[\Gamma_1,\Gamma_2,\ldots,\Gamma_q]$ satisfies
\begin{equation}\label{l1}
Char_A(\Gamma[\Gamma_1,\Gamma_2,\ldots,\Gamma_q],x)
=
Char(\mathcal{Q},x)
\prod_{i=1}^{q}
\frac{Char(\Gamma_i,x)}{x-r_i}.
\end{equation}
\end{thm}

In the above theorem, if we replace $x-r_i$ by $x- deg(v_i)+r_i$, we get the characteristic polynomial $Char_L(\Gamma[\Gamma_1,\Gamma_2, \ldots , \Gamma_q],x)$ of the Laplacian matrix. Similarly, if we replace $x-r_i$ by $x- deg(v_i)-r_i$, we get the characteristic polynomial $Char_Q(\Gamma[\Gamma_1,\Gamma_2, \ldots , \Gamma_q],x)$ of a signless Laplacian matrix. That is,
\begin{equation}\label{l2}
Char_L(\Gamma[\Gamma_1,\Gamma_2, \ldots , \Gamma_q],x) = Char(\mathcal{Q},x)  \prod_{i=1}^{q} \frac{Char(\Gamma_i, x)}{(x- deg(v_i)+r_i)},
\end{equation}
\begin{equation}\label{l3}
Char_Q(\Gamma[\Gamma_1,\Gamma_2, \ldots , \Gamma_q],x) = Char(\mathcal{Q},x)  \prod_{i=1}^{q} \frac{Char(\Gamma_i, x)}{(x- deg(v_i)-r_i)},
\end{equation}
where $d(v_i)$ is the degree of the vertex $v_i \in V(\Gamma[\Gamma_1,\Gamma_2, \ldots, \Gamma_q])$ and $Char(\Gamma_i, x)$ is the characteristic polynomial of the corresponding diagonal block matrix of the matrix associated to the graph $\Gamma[\Gamma_1,\Gamma_2, \ldots, \Gamma_q]$.

\section{Power graph associated to the metacyclic group ${\Omega}_{2,2^kq,2^{k-1}q-1}$}\label{Sec3}
This section is devoted to the study of metacyclic groups, including a representation of the power graph of a split metacyclic group as a joined union graph.

A group $\Omega$ is a metacyclic group if it contains a cyclic normal subgroup $N$ such that the quotient group $\Omega/N$ is also cyclic. A classical result of H\"{o}lder (see \cite[Theorem 21]{Zassenhaus(1999)}) states that a group $\Omega$ is metacyclic if and only if it admits the presentation
$$\Omega_{a, b,l, m} = \langle r, s  : s^{a} = r^{b}, r^{l} = e, srs^{-1}= r^{m}\rangle, \ \text{where} \ a, b, l, m \in \mathbb{N}, \ l |(m^a-1) \ \text{and} \ l | b(m-1).$$
In the special case $b=l$, the group $\Omega_{a,b,l,m}$ is referred to a split metacyclic group and is written as
$$ \Omega_{a, b,m} = \langle r, s  : s^{a} = r^{b} = e, srs^{-1}= r^{m}\rangle, \ \text{with} \ l |(m^b-1).$$

Note that $\Omega_{2, n,n-1} \cong D_{2n}$,  $\Omega_{2,4n,2n-1} \cong SD_{8n}$, and $\Omega_{2,2n,n,2n-1} \cong Q_{4n}$, i.e., $D_{2n}$, $SD_{8n}$, $Q_{4n}$ arise as special cases of $\Omega_{a,b,l,m}.$ Here, we study a particular family of split metacyclic groups denoted by $\mathcal{G}$, where $\mathcal{G}$ is a non-commutative group of order $2^{k+1}q$ with $k > 2$ and $q\neq 2$ prime. The group $\mathcal{G}$ is given as $$\mathcal{G} = {\Omega}_{2,2^kq,2^{k-1}q-1} = \langle r,s : s^{2} = r^{2^kq} = e,  srs^{-1} = r^{2^{k-1}q-1}\rangle.$$

This section deals with metacyclic groups and describes the power graph of a split metacyclic group through a joined union graph construction.

\begin{prop}\label{p3.1}
Let $\mathcal{G}$ be the split metacyclic group of order $2^{k+1}q$, where $k\geq2$ and $q \neq 2$ is prime. Then the power graph of $\mathcal{G}$ is 
    $$P(\mathcal{G}) \cong \Gamma[K_1, K_1, K_{\phi(2^{k}q)}, K_{\phi(2^{k-1}q)},\ldots,K_{\phi(2q)},K_{\phi(q)},K_{\phi(2^2)},\ldots,K_{\phi(2^k)}, \undermat{2^{k-2}q} {K_2, \ldots, K_2}, \overline{K}_{2^{k-1}q}]$$ with respect to $\Gamma$ given in the following Figure 1.
    
\vspace{.15cm}
\hspace{3.4cm}
\begin{tikzpicture}[thick, scale=.95] 
\setlength{\unitlength}{1mm}
\filldraw[color=black!60, fill=black!5, very thick](2,0);

\draw (-.55,-1.5) node { \scriptsize $\textcolor{black}{\bullet}$};
\draw (-.5,2) node { \scriptsize $\textcolor{black}{\bullet}$};
\draw (-2.5,-2) node { \scriptsize $\textcolor{black}{\bullet}$};
\draw (-1.2,-3) node { \scriptsize $\textcolor{brown}{\bigstar}$};
\draw (-.4,-3) node { \scriptsize $\textcolor{brown}{\bigstar}$};
\draw (0,-3) node { \scriptsize $\textcolor{brown}{\bigstar}$};
\draw (.4,-3) node { \scriptsize $\textcolor{brown}{\bigstar}$};
\draw (1.3,-3) node { \scriptsize $\textcolor{brown}{\bigstar}$};
%~~~~~~~~~~~~~~~~~~~~~~~~~~~~~~~~~~~~~~~~~~~~~~~~~~~~
\draw (-3,0) node { \scriptsize $\textcolor{red}{\bigstar}$};
\draw (-3,1.5) node { \scriptsize $\textcolor{red}{\bigstar}$};
\draw (-3,2.3) node { \scriptsize $\textcolor{red}{\bigstar}$};
\draw (-3,2.8) node { \scriptsize $\textcolor{red}{\bigstar}$};
\draw (-3,3.3) node { \scriptsize $\textcolor{red}{\bigstar}$};
\draw (-3,4) node { \scriptsize $\textcolor{red}{\bigstar}$};
\draw (-.5,4.3) node { \scriptsize $\textcolor{black}{\bullet}$};
\draw (2,3.1) node { \scriptsize $\textcolor{blue}{\bigstar}$};
\draw (2,1.55) node { \scriptsize $\textcolor{blue}{\bigstar}$};
\draw (2,.95) node { \scriptsize $\textcolor{blue}{\bigstar}$};
\draw (2,.5) node { \scriptsize $\textcolor{blue}{\bigstar}$};
\draw (2,0) node { \scriptsize $\textcolor{blue}{\bigstar}$};
\draw (2,-1) node { \scriptsize $\textcolor{blue}{\bigstar}$};
%~~~~~~~~~~~~~~~~~~~~~~~~~~~~~~~~~~~~~~~~~~~~~~~~~~~~
\draw (-.8, -1.6) node {\scriptsize $\textcolor
	{black}{K_1}$};
\draw (-.28,2.6) node {\scriptsize $\textcolor
	{black}{{K}_1}$};
\draw (-2.9, -2) node {\scriptsize $\textcolor
	{black}{\overline{K}_{2^{k-1}q}}$};
\draw (1,-3.2) node {\scriptsize $\textcolor
	{black}{K_2}$};
 \draw (-1.5,-3.2) node {\scriptsize $\textcolor
	{black}{K_2}$};
\draw (-3.63,0) node {\scriptsize $\textcolor
	{black}{K_{\phi(2^{k}q)}}$};
\draw (-3.85,1.5) node { \scriptsize $\textcolor
	{black}{K_{\phi(2^{k-1}q)}}$};
\draw (-3.6,4) node { \scriptsize $\textcolor
	{black}{K_{\phi(2q)}}$};
 \draw (-.5,4.5) node { \scriptsize $\textcolor
	{black}{K_{\phi(q)}}$};
\draw (2.65,3.1) node { \scriptsize $\textcolor
	{black}{K_{\phi(2^2)}}$};
\draw (2.6,1.55) node { \scriptsize $\textcolor{black}{K_{\phi(2^3)}}$};
\draw (2.6, -1) node { \scriptsize $\textcolor
	{black}{K_{\phi(2^{k})}}$};
%~~~~~~~~~~~~~~~~~~~~~~~~~~~~~~~~~~~~~~~~~~~~~~~~~~~~~~
\draw (-.5,2)--(-3,0);
\draw (-.5,2)--(-3,1.5);
\draw (-.5,2)--(-3,4);
\draw (-.5,2)--(-.5,4.3);
\draw (-.5,2)--(2,3.1);
\draw (-.5,2)--(2,1.55);
\draw (-.5,2)--(2,-1);
\draw (-.5,2)--(1.3,-3);
\draw (-.5,2)--(-1.2,-3);
\draw (-.5,2)--(-2.5,-2);
\draw (-.5,2)--(-.55,-1.5);
\draw (-3,0) to [bend left=30] (-3,4);
\draw (-3,1.5) to [bend left=30] (-3,4);
%~~~~~~~~~~~~~~~~~~~~~~~~~~~~~~~~~~~~~~~~~~~~~~~~~~~~~~
\draw (-3,0)--(-.5,4.3);
\draw (-3,0)--(2,3.1);
\draw (-3,0)--(2,1.55);
\draw (-3,0)--(2,-1);
%~~~~~~~~~~~~~~~~~~~~~~~~~~~~~~~~~~~~~~~~~~~~~~~~~~~~~~
\draw (-3,1.5)--(-.5,4.3);
\draw (-3,1.5)--(2,3.1);
\draw (-3,1.5)--(2,1.55);
\draw (-3,4)--(-.5,4.3);
%~~~~~~~~~~~~~~~~~~~~~~~~~~~~~~~~~~~~~~~~~~~~~~~~~~~~~~
\draw (2,3.1)--(2,1.55);
\draw (2,3.1) to [bend left=30] (2,-1);
\draw (2,1.55) to [bend left=30] (2,-1);
%~~~~~~~~~~~~~~~~~~~~~~~~~~~~~~~~~~~~~~~~~~~~~~~~~~~~~~
\draw (-.55,-1.5)--(-3,0);
\draw (-.55,-1.5)--(-3,1.5);
\draw (-.55,-1.5)--(-3,4);
\draw (-.55,-1.5)--(2,3.1);
\draw (-.55,-1.5)--(2,1.55);
\draw (-.55,-1.5)--(2,-1);
\draw (-.55,-1.5)--(-1.2,-3);
\draw (-.55,-1.5)--(1.3,-3);
%\draw [dashed] (-1.5,-3)--(1,-3);
\draw [dashed] (-.55,-1.5)--(-.4,-3);
\draw [dashed] (-.55,-1.5)--(0,-3);
\draw [dashed] (-.55,-1.5)--(.4,-3);
\draw [dashed] (-.5,2)--(-.4,-3);
\draw [dashed] (-.5,2)--(0,-3);
\draw [dashed] (-.5,2)--(.4,-3);
%\draw [dashed] (2,1.55)--(2,.95) ;
%\draw [dashed] (2,.95)--(2,.5) ;
%\draw [dashed] (2,.5)--(2,0) ;
%\draw [dashed] (2,0)--(2,-1);
 
%~~~~~~~~~~~~~~~~~~~~~~~~~~~~~~~~~~~~~~~~~~~~~~~~~~~~~~
\draw (-3,0)--(-3,1.5);
\draw[dashed] (-3,1.5)--(-3,4);
\draw[dashed] (2,1.55)-- (2,-1);
\draw [decorate, decoration={brace, amplitude=8pt, mirror}, thick] 
  (-1.4,-3.4) -- (1.5,-3.4) node [black, midway, yshift=-13pt] {\scriptsize $V_1$};
  
  \draw [decorate, decoration={brace, amplitude=8pt}, thick] 
  (-4.5,-0.2) -- (-4.5,4.2) node [black, midway, xshift=-13pt] {\scriptsize $V_2$};

\draw [decorate, decoration={brace, amplitude=8pt}, thick] 
  (3.1,3.2) -- (3.1,-1.2) node [black, midway, xshift=13pt] {\scriptsize $V_3$};
\draw (-.6, -4.4) node { \small \textbf{Figure 1}: Graph $\Gamma$};
\end{tikzpicture} \label{fig1}

Here, the brown star, red star, and blue star represent vertices of the graph $\Gamma$ belong to the set $V_1 = \{{K_2, K_2, \ldots, K_2, (2^{k-2}q, \text{times})}\}$, $V_2=\{K_{\phi(2^{k}q)}, K_{\phi(2^{k-1}q)},\ldots, K_{\phi(2q)}\}$, and $V_3 =\{K_{\phi(2^2)}$ ,$K_{\phi(2^3)},\ldots, K_{\phi(2^k)}\}$ respectively.
\end{prop}

\noindent{\textbf{Proof.}} Consider the subsets $\mathcal{H}_1=\langle r\rangle,$
$\mathcal{H}_2=\{sr^{2t}:1\leq t\leq 2^{k-1}q\},$ and $\mathcal{H}_3=\{sr^{2j+1}:0\leq j\leq 2^{k-1}q-1\}$ of $\mathcal{G}$. Since $\mathcal{H}_1=\langle r\rangle \cong \mathbb{Z}_{2^kq},$ the subgroup $\mathcal{H}_1$ contains exactly $\phi(d)$ elements of order $d$ for every divisor $d$ of $2^kq$. Let $d_i$, $1\leq i\leq 2k+2$, denote the divisors of $2^kq$. In the power graph $P(\mathcal{H}_1)\cong P(\mathbb{Z}_{2^kq})$, two distinct vertices $r^i$ and $r^j$ are adjacent whenever $o(r^i)|o(r^j)$ or $o(r^j)|o(r^i)$. Consequently, for each divisor $d_i$, the vertices of order $d_i$ induce a complete subgraph of size $\phi(d_i)$. Moreover, if $d_i\mid d_j$, then every vertex of order $d_i$ is adjacent to every vertex of order $d_j$. Now consider the set $\mathcal{H}_2$. Since $o(sr^t)=2$ for even $t$, each element of $\mathcal{H}_2$ is adjacent only to the identity element $e$. Hence, these vertices produce copies of $K_2$. For odd $t$, we have $o(sr^t)=4.$ Also, $\langle r^{2^{k-1}q}\rangle \subset \langle sr^t\rangle.$ Each subgroup $\langle sr^t\rangle$ therefore has exactly $\phi(4)=2$ generators, and these generators form a copy of $K_2$. Every such $K_2$ is adjacent to both $e$ and $r^{2^{k-1}q}$. Since there are $2^{k-1}q$ odd integers between $0$ and $2^kq-1$, and each subgroup contributes two generators, we obtain $2^{k-2}q$ distinct copies of $K_2$. Combining these observations, the power graph $P(\mathcal{G})$ can be expressed as
\[
\Gamma[K_1,K_1,K_{\phi(2^kq)},K_{\phi(2^{k-1}q)},\ldots,K_{\phi(2q)},
K_{\phi(q)},K_{\phi(2^2)},K_{\phi(2^3)},\ldots,K_{\phi(2^k)},
\undermat{2^{k-2}q}{K_2,\ldots,K_2},\overline{K}_{2^{k-1}q}],
\]
where $\Gamma$ is the graph shown in Figure 1. \hfill$\Box$

\section{Characteristic Polynomial of the $P(\mathcal{G})$} \label{Sec4}
In the present section, we compute the characteristic polynomials of several matrices related to $P(\mathcal{G})$.

\begin{thm}\label{t4.1}
The characteristic polynomial of the adjacency matrix of $P(\mathcal{G})$ is given by
\[ Char(A,x)
=
x^{2^{k-1}q-1}
(x-1)^{2^{k-2}q-1}
(x+1)^{5\cdot 2^{k-2}q - 2k - 2}Char(\mathcal{Q},x),
\]
where $Char(\mathcal{Q},x)$ denotes the characteristic polynomial of the quotient matrix $\mathcal{Q}$. The matrix $\mathcal{Q}$ has the block form 
$$\begin{bmatrix}
 \begin{array}{c|cccccccccc|c}
   \multirow{2}{*}{$A_1$} & \phi(2^kq) & \phi(2^{k-1}q) &\cdots & \phi(2^2q) & \phi(2q) & \phi(q)  & 2  &2^2  & \cdots & \phi(2^k) & \multirow{2}{*}{$A_2$} \\
  & \phi(2^kq)  & \phi(2^{k-1}q)  &\cdots & \phi(2^2q)  & \phi(2q) & 0 & 2  & 2^2  & \cdots & \phi(2^k) & \\ \hline
  \multirow{12}{*}{$B^T$} & a_1 & \phi(2^{k-1}q)  &\cdots & \phi(2^2q) & \phi(2q) & \phi(q) & 2 & 2^2 & \cdots & \phi(2^k) & \multirow{12}{*}{$O^T$} \\
  & \phi(2^kq) & a_2 & \cdots & \phi(2^2q) & \phi(2q) & \phi(q) & 2 & 2^2 & \cdots & 0 & \\
  & \vdots & \vdots & \ddots & \vdots & \vdots & \vdots & \vdots & \vdots & \ddots & \vdots & \\
  & \phi(2^{k}q) & \phi(2^{k-1}q) &\cdots & a_{k-2} & \phi(2q) & \phi(q) & 2 & 0 & \cdots & 0 & \\
  & \phi(2^{k}q) & \phi(2^{k-1}q) &\cdots & \phi(2^2q) & a_{k-1} & \phi(q) & 0 & 0 & \cdots & 0 & \\
  & \phi(2^{k}q) & \phi(2^{k-1}q) &\cdots & \phi(2^2q) & \phi(2q) & a_{k} & 0 & 0 & \cdots & 0 & \\
 & \phi(2^{k}q) & \phi(2^{k-1}q) &\cdots & \phi(2^2q) & 0 & 0 & a_{k+1} & 2^2 & \cdots & \phi(2^k) & \\
& \phi(2^{k}q) & \phi(2^{k-1}q) &\cdots & 0 & 0 & 0 & 2 & a_{k+2} & \cdots & \phi(2^k) &  \\
 & \vdots & \vdots & \ddots & \vdots & \vdots & \vdots & \vdots & \vdots & \ddots & \vdots & \\
    % 2^{k-2} &  J & J & J & J &\cdots & 0 & 0 & 0 & J & J & \cdots & J \\
    & \phi(2^{k}q) & 0 &\cdots & 0 & 0 & 0 & 2 & 2^2 & \cdots & a_{2k} & \\ \hline
  \multirow{2}{*}{$A_3$} & 0 & 0 &\cdots & 0 & 0 & 0 & 0 & 0 & \cdots & 0 &\multirow{2}{*}{$A_4$} \\
   & 0 & 0 &\cdots & 0 & 0 & 0 & 0 & 0 & \cdots & 0 & \\
\end{array}
 \end{bmatrix},$$
where, $A_1=\begin{bmatrix}
    0 & 1\\
    1 & 0
\end{bmatrix}, A_2= \begin{bmatrix}
    2^{k-1}q & 2^{k-1}q\\
    2^{k-1}q & 0
\end{bmatrix}, B_{2 \times 2k}=\begin{bmatrix}
    1 & 1 & \cdots & 1 & 1 & 1 & 1 & 1 & \cdots & 1\\
    1 & 1 & \cdots & 1 & 1 & 0 & 1 & 1 & \cdots & 1 
\end{bmatrix},$ $A_3=\begin{bmatrix}
    1 & 1\\
    1 & 0
\end{bmatrix}, A_4=\begin{bmatrix}
    1 & 0\\
    0 & 0
\end{bmatrix},$
  $a_1=\phi(2^kq)-1, a_2=\phi(2^{k-1}q)-1, \ldots, a_{k-2}=\phi(2^2q)-1, a_{k-1}=\phi(q)-1, a_{k}=\phi(q)-1,$ $a_{k+1}=1, a_{k+2}=3, \ldots, a_{2k}=2^{k-1}-1.$
\end{thm}

\noindent{\textbf{Proof.}} Consider the subsets $H_1 = \langle r \rangle,
H_2 = \{ sr^{2t} : 1 \leq t \leq 2^{k-1}q \}, \text{ and } H_3 = \{ sr^{2j+1} : 0 \leq j \leq 2^{k-1}q - 1 \}$ of the group $\mathcal{G}$. Since $H_1 = \langle r \rangle \cong \mathbb{Z}_{2^k q}$, it follows that for each divisor $d \mid 2^k p$, the number of elements of order $d$ in $H_1$ is $\phi(d)$, where $\phi$ denotes Euler's totient function. In particular, $H_1$ contains exactly one element of order $1$, one element of order $2$, $\phi(2^i)$ elements of order $2^i$ for $2 \leq i \leq k$, $\phi(q)$ elements whose order is $q$, and $\phi(2^i q)$ elements whose order is $2^i q$ for $1 \leq i \leq k$.

To define the adjacency matrix $A$ of $P(\mathcal{G})$, an ordering of the group elements is fixed for the rows and columns. The indexing starts with the identity element and the unique element of order $2$. Thereafter, the elements having orders $2^{d}q$, where $d=1,2,\ldots,k$, are arranged in decreasing order of their orders. The elements of order $q$ are then included. Next, the elements with orders $2^{t}$, for $t=2,3,\ldots,k$, are placed in increasing order. Finally, the elements belonging to $H_3$ and $H_2$ are listed successively. With this arrangement of vertices, the adjacency matrix $A$ takes the following form.

$$A = \begin{bmatrix}
    A_{11} & A_{12}\\
    A^t_{12} & A_{22}
\end{bmatrix}_{2^{k+1}}, \text{ where }$$
$$ A_{12} = [a_{ij}]_{2^kq} = \begin{cases}
    1, & \text{if $i = 1$ and $2^{k}q+1 \leq j \leq 2^{k+1}q$} \\
    1, & \text{if $i=2$ and $2^{k}q+1 \leq j \leq 3\cdot2^{k-1}q+1$}\\
    0, & \text{otherwise}
    \end{cases}, $$
    
\[ \hspace{-1cm} A_{11} = 
 \begin{blockarray}{cccccccccccc}
 \begin{block}{[cccccccccccc]}
   0 & 1 &  \mathcal{J} &  \mathcal{J} &\cdots &  \mathcal{J} &  \mathcal{J} &  \mathcal{J} &  \mathcal{J} &  \mathcal{J} & \cdots &  \mathcal{J} \\
  1 & 0 &  \mathcal{J} &  \mathcal{J} &\cdots &  \mathcal{J} &  \mathcal{J} & 0 &  \mathcal{J} &  \mathcal{J} & \cdots &  \mathcal{J} \\
    \mathcal{J} & \mathcal{J} & A_3 &  \mathcal{J} &\cdots & \mathcal{J} & \mathcal{J} & \mathcal{J} & \mathcal{J} & \mathcal{J} & \cdots & \mathcal{J} \\
   \mathcal{J} & \mathcal{J} & \mathcal{J} & A_4 &\cdots & \mathcal{J} & \mathcal{J} & \mathcal{J} & \mathcal{J} & \mathcal{J} & \cdots & 0 \\
 \vdots & \vdots & \vdots & \vdots & \ddots & \vdots & \vdots & \vdots & \vdots & \vdots & \ddots & \vdots\\
 \mathcal{J} &  \mathcal{J} & \mathcal{J} & \mathcal{J} &\cdots & A_t & \mathcal{J} & \mathcal{J} &  \mathcal{J} & 0 & \cdots & 0 \\
   \mathcal{J} & \mathcal{J} &  \mathcal{J} &  \mathcal{J} &\cdots & \mathcal{J} & A_{t+1} &  \mathcal{J} & 0 & 0 & \cdots & 0 \\
   \mathcal{J} & 0 &  \mathcal{J} &  \mathcal{J} &\cdots &  \mathcal{J} & \mathcal{J} & A_{t+2} & 0 & 0 & \cdots & 0 \\
   \mathcal{J} & \mathcal{J} &  \mathcal{J} &  \mathcal{J} &\cdots &  \mathcal{J} & 0 & 0 & A_{t+3} & \mathcal{J} & \cdots & \mathcal{J} \\
  \mathcal{J} & \mathcal{J} & \mathcal{J} & \mathcal{J} &\cdots & 0 & 0 & 0 & \mathcal{J} & A_{t+4} & \cdots & \mathcal{J} \\
 \vdots & \vdots & \vdots & \vdots & \ddots & \vdots & \vdots & \vdots & \vdots & \vdots & \ddots & \vdots\\
    \mathcal{J} & \mathcal{J} & \mathcal{J} & 0 &\cdots & 0 & 0 & 0 & \mathcal{J} & \mathcal{J} & \cdots & A_{2(k+1)} \\
\end{block}
 \end{blockarray}_{2^kq},\]

  $$A_{22} = 
 \begin{bmatrix}
       0 & 1 &0 & 0 & \cdots & 0 & 0 & 0 & \cdots & 0 \\
   1  & 0 &0 &0 &\cdots & 0 & 0 & 0 & \cdots & 0 \\
 0& 0 & 0 & 1 & \cdots & 0 & 0 & 0 & \cdots & 0 \\
  0& 0 & 1 & 0 & \cdots & 0 & 0 & 0 & \cdots & 0 \\
  \vdots & \vdots & \vdots & \vdots & \ddots & \vdots & \vdots & \vdots & \ddots & \vdots \\
   0 & 0 & 0 & 0 & \cdots & 0 & 1 & 0 & \cdots & 0\\
 0 & 0 & 0 & 0 & \cdots & 1 & 0 & 0 & \cdots & 0 \\
    0 & 0 & 0 & 0 & \cdots & 0 & 0 & 0 & \cdots & 0\\
 \vdots & \vdots & \vdots & \vdots & \ddots & \vdots & \vdots & \vdots & \ddots & \vdots \\
   0 & 0 & 0 & 0& \cdots & 0 & 0 & 0 & \cdots & 0 \\
\end{bmatrix}_{2^kq}, A_{m} = [a_{ij}] = \begin{cases}
    0 , & \text{if $i = j$}\\
    1, & \text{otherwise}
    \end{cases},$$ where $3\leq m \leq 2(k+1).$

Now, consider the equitable partition $\pi  = \{V_1, V_2, \ldots, V_{2k+4}\}$ of $V(P(\mathcal{G}))$ as $V_1 = \{e\}, V_2 = \{ r^{2^{k-1}q}\}, V_3 = \{r^{i}: \text{ gcd } (r^{i},2^{k}q) = 1\}, V_4 = \{r^{i}:  o(r^{i}) = 2^{k-1}q\}, \ldots, V_j =\{r^{i} :  o(r^{i}) = 2^{2}q\}, V_{j+1} = \{r^{i} :  o(r^{i}) = 2q\}, V_{j+2} = \{r^{i} :  o(r^{i}) = q\}, V_{j+3} = \{r^{i} : o(r^{i}) = 2^2\}, V_{j+4} \{r^{i} :  o(r^{i}) = 2^3 \}, \ldots, V_{2{(k+1)}} = \{r^{i} : o(r^{i}) = 2^{k}\}, V_{2k+3} =\{sr^{2j+1} : 0 \leq j \leq 2^{k-1}q-1\}, \text{ and }  V_{2k+4} = \{sr^{2t} : 1 \leq t \leq 2^{k-1}q \}$. Using Proposition~\ref{p3.1} together with Equation~(\ref{l1}) of Theorem~\ref{t2.1}, the characteristic polynomial of the adjacency matrix associated with $P(\mathcal{G})$ is obtained as follows: $$ Char(A,x) = Char(\mathcal{Q},x)  \prod_{i=1}^{2k+4} \frac{Char(H_i, x)}{(x-r_i)},$$ where, for each $1 \leq i \leq 2k+2$, the graph $H_i$ is the block matrix corresponding to a complete subgraph defined on the vertex set $V_i$, $H_{2k+3} =\underbrace{K_2 \cup K_2 \cup \cdots \cup K_2}_{2^{k-2}q}$ with the vertex set $V_{2k+3}$ and edge set $E = \{sr\, sr^{2^{k-1}q+1},sr^3 \,sr^{2^{k-1}q+3}$, \ldots, $sr^{2^{k-1}q-1} \, sr^{2^{k}q-1}\}$, $H_{2k+4}=\overline{K}_{2^{k-1}q}$ with the vertex set $V_{2k+4}$. The matrix $\mathcal{Q}$ denotes the quotient matrix associated with the partition introduced above. This gives
 $$Char(A,x) = x^{2^{k-1}q-1}(x-1)^{2^{k-2}q-1}(x+1)^{5\cdot2^{k-2}q-2k-2} Char(\mathcal{Q},x). \hspace{4.5cm} \hfill \Box$$ 

The following corollary specialises the above theorem to the case $k=2$ with $q$ an odd prime.
\begin{cor}\label{c4.1}
In the particular case $k=2$ with $q$ chosen as an odd prime, the characteristic polynomial $Char(A,x)$ is
\newline
$x^{2q-1}(x+1)^{5q-6}(x-1)^{q-1}\big[x^8 - (4q-6)x^7 - (20q-7)x^6 + (35q^2-71q-6)x^5+ (110q^2-112q-15) x^4 - (80q^3-238q^2+102q+2)x^3 -(158q^3-258q^2+62q-7)x^2 + (48q^4-168q^3+127q^2-23q+2)x + (40q^4-50q^3-24q^2+18q)\big].$
\end{cor}

\noindent{\textbf{Proof.}}  Applying Theorem~\ref{t4.1} to the case $k=2$, we obtain the following characteristic polynomial for the adjacency matrix of $P(\mathcal{G})$ is $Char(A,x) =x^{2q-1}(x+1)^{5q-6}(x-1)^{q-1}Char(\mathcal{Q},x),$  where 
$$\mathcal{Q} = \begin{bmatrix}
    0 & 1 & 2(q-1) & q-1 & q-1 & 2 & 2q & 2q\\
    1 & 0 & 2(q-1) & q-1 & 0 & 2 & 2q & 0\\
    1 & 1 & 2q-3 & q-1 & q-1 & 2 & 0 & 0\\
    1 & 1 & 2(q-1) & q-2 & q-1 & 0 & 0 & 0\\
    1 & 0 & 2(q-1) & q-1 & q-2 & 0 & 0 & 0\\
    1 & 1 & 2(q-1) & 0 & 0 & 1 & 0 & 0\\
     1 & 1 & 0 & 0 & 0 & 0 & 1 & 0\\
     1 & 0 & 0 & 0 & 0 & 0 & 0 & 0\\
\end{bmatrix}.$$
By direct computation, one obtains $\det(xI-\mathcal{Q})=Char(\mathcal{Q},x) = \big[x^8 - (4q-6)x^7 - (20q-7)x^6 + (35q^2-71q-6)x^5+ (110q^2-112q-15) x^4 - (80q^3-238q^2+102q+2)x^3 -(158q^3-258q^2+62q-7)x^2 + (48q^4-168q^3+127q^2-23q+2)x + (40q^4-50q^3-24q^2+18q)\big].$ 

Now, by using $Char(\mathcal{Q},x)$  in $Char(A,x)$, which completes the proof.\hfill$\Box$

\begin{thm}\label{t4.2}
The Laplacian characteristic polynomial of $P(\mathcal{G})$ is \newline
$Char(L,x)=(x-1)^{2^{k-1}q}(x-2)^{2^{k-2}q-1}(x-4)^{2^{k-2}q}(x-2^{k}q)^{2^{k-1}(q-1)-1}(x-2^{k-1}(2q-1))^{2^{k-2}(q-1)-1}(x-2^{k}(q-1)-4)^{2q-3}(x-2^{k}(q-1)-2)^{q-1}(x-2^{k}(q-1)-1)^{q-2}(x-2^{k}q+2q-2)(x-2^{k}q-4q-4)^3(x-2^{k-1}(q+1))^{2^{k-1}-1}Char(\mathcal{Q}_1,x),$ where $Char(\mathcal{Q}_1,x)$ denotes the characteristic polynomial of the quotient matrix $\mathcal{Q}_1$ associated with the given equitable partition. The matrix $\mathcal{Q}_1$ has the following block representation:
$$\begin{bmatrix}
 \begin{array}{c|cccccccccc|c}
   \multirow{2}{*}{$\mathcal{L}_1$} & -\phi(2^kq) & -\phi(2^{k-1}q) &\cdots & -\phi(2^2q) & -\phi(2q) & -\phi(q)  & -2  & -2^2  & \cdots & -\phi(2^k) & \multirow{2}{*}{$\mathcal{L}_2$} \\
  & -\phi(2^kq)  & -\phi(2^{k-1}q)  &\cdots & -\phi(2^2q) & -\phi(2q) & 0 & -2  & -2^2  & \cdots & -\phi(2^k) &  \\ \hline
  \multirow{12}{*}{$C^T$}  & l_3 & -\phi(2^{k-1}q)  &\cdots & -\phi(2^2q) & -\phi(2q) & -\phi(q) & -2 & -2^2 & \cdots & -\phi(2^k) & \multirow{12}{*}{$O^T$} \\
   & -\phi(2^kq) & l_4 &\cdots & -\phi(2^2q) & -\phi(2q) & -\phi(q) & -2 & -2^2 & \cdots & 0 & \\
  & \vdots & \vdots & \ddots & \vdots & \vdots & \vdots & \vdots & \vdots & \ddots & \vdots & \\
   & -\phi(2^{k}q) & -\phi(2^{k-1}q) &\cdots & l_t & -\phi(2q) & -\phi(q) & -2 & 0 & \cdots & 0 & \\
 & -\phi(2^{k}q) & -\phi(2^{k-1}q) &\cdots & -\phi(2^2q) & l_{t+1} & -\phi(q) & 0 & 0 & \cdots & 0 \\
   & -\phi(2^{k}q) & -\phi(2^{k-1}q) &\cdots & -\phi(2^2q) & -\phi(2q) & l_{t+2} & 0 & 0 & \cdots & 0 & \\
  & -\phi(2^{k}q) & -\phi(2^{k-1}q) &\cdots & -\phi(2^2q) & 0 & 0 & l_{t+3} & -2^2 & \cdots & -\phi(2^k) & \\
& -\phi(2^{k}q) & -\phi(2^{k-1}q) &\cdots & 0 & 0 & 0 & -2 & l_{t+4} & \cdots & -\phi(2^k) & \\
  & \vdots & \vdots & \ddots & \vdots & \vdots & \vdots & \vdots & \vdots & \ddots & \vdots &\\
    % 2^{k-2} &  J & J & J & J &\cdots & 0 & 0 & 0 & J & J & \cdots & J \\
    & -\phi(2^{k}q) & 0 &\cdots & 0 & 0 & 0 & -2 & -2^2 & \cdots & l_{2(k+1)} & \\ \hline
  \multirow{2}{*}{$\mathcal{L}_3$} & 0 & 0 &\cdots & 0 & 0 & 0 & 0 & 0 & \cdots & 0 & \multirow{2}{*}{$\mathcal{L}_4$} \\
   & 0 & 0 &\cdots & 0 & 0 & 0 & 0 & 0 & \cdots & 0 & \\
\end{array}
 \end{bmatrix},$$
 where, $\mathcal{L}_1=\begin{bmatrix}
    l_1 & -1\\
    -1 & l_2
\end{bmatrix}, \mathcal{L}_2= \begin{bmatrix}
    -2^{k-1}q & -2^{k-1}q\\
    -2^{k-1}q & 0
\end{bmatrix}, \mathcal{L}_3=\begin{bmatrix}
    -1 & -1\\
    -1 & 0
\end{bmatrix}, \mathcal{L}_4=\begin{bmatrix}
    2 & 0\\
    0 & 1
\end{bmatrix},$ $C_{2 \times 2k}=\begin{bmatrix}
    -1 & -1 & \cdots & -1 & -1 & -1 & -1 & -1 & \cdots & -1\\
    -1 & -1 & \cdots & -1 & -1 & 0 & -1 & -1 & \cdots & -1 
\end{bmatrix},$ $l_1 = (2^{k+1}q-1), l_2 = (3 \cdot 2^{k-1}-1)q, l_3 =2^{k-1}(q+1),  l_4 = 2^{k-2}(3q-1), \ldots, l_t = 2^{k}(q-1)-2q+7, l_{t+1} =2^{k}(q-1)-q+3$, $l_{t+2}=2^{k}(q-1)-q+2, l_{t+3} = 2^{k}q-2q, l_{t+4} = 2^{k}q-4q, \ldots, l_{2(k+1} = 2^{k-1}q.$
\end{thm}

\noindent{\textbf{Proof.}} By the same indexing of the vertices as in Theorem \ref{t4.1}, the Laplacian matrix of $P(\mathcal{G})$ is
$$L = \begin{bmatrix}
    L_{11} & L_{12}\\
    L^t_{12} & L_{22}
\end{bmatrix}, \text{where}$$  $$ L_{12} = [l_{ij}]_{2^kq} = \begin{cases}
    -1, & \text{if $i = 1$ and $2^{k}q+1 \leq j \leq 2^{k+1}q$} \\
    -1, & \text{if $i=2$ and $2^{k}q+1 \leq j \leq 3\cdot2^{k-1}q+1$}\\
    0, & \text{otherwise}
    \end{cases},$$
\[ \hspace{-1cm} L_{11} = 
 \begin{blockarray}{cccccccccccc}
 \begin{block}{[cccccccccccc]}
    e_1 & -1 & -\mathcal{J} & -\mathcal{J} &\cdots & -\mathcal{J} & -\mathcal{J} & -\mathcal{J} & -\mathcal{J} & -\mathcal{J} & \cdots & -\mathcal{J} \bigstrut[t] \\
  -1 & e_2 &  -\mathcal{J} & -\mathcal{J} &\cdots & -\mathcal{J} & -\mathcal{J} & 0 & -\mathcal{J} &  -\mathcal{J} & \cdots &  -\mathcal{J} \\
    -\mathcal{J} & -\mathcal{J} & L_3 &  -\mathcal{J} &\cdots & -\mathcal{J} & -\mathcal{J} & -\mathcal{J} &  -\mathcal{J} & -\mathcal{J} & \cdots &  -\mathcal{J} \\
   -\mathcal{J} &  -\mathcal{J} &  -\mathcal{J} & L_4 &\cdots &  -\mathcal{J} & -\mathcal{J} &  -\mathcal{J} &  -\mathcal{J} & -\mathcal{J} & \cdots & 0 \\
 \vdots & \vdots & \vdots & \vdots & \ddots & \vdots & \vdots & \vdots & \vdots & \vdots & \ddots & \vdots\\
  -\mathcal{J} & -\mathcal{J} & -\mathcal{J} & -\mathcal{J} &\cdots & L_t & -\mathcal{J} & -\mathcal{J} & -\mathcal{J} & 0 & \cdots & 0 \\
  -\mathcal{J} & -\mathcal{J} & -\mathcal{J} & -\mathcal{J} &\cdots & -\mathcal{J} & L_{t+1} & -\mathcal{J} & 0 & 0 & \cdots & 0 \\
  -\mathcal{J} & 0 &  -\mathcal{J} & -\mathcal{J} &\cdots & -\mathcal{J} & -\mathcal{J} & L_{t+2} & 0 & 0 & \cdots & 0 \\
  -\mathcal{J} & -\mathcal{J} & -\mathcal{J} & -\mathcal{J} &\cdots & -\mathcal{J} & 0 & 0 & L_{t+3} & -\mathcal{J} & \cdots & -\mathcal{J} \\
 -\mathcal{J} & -\mathcal{J} & -\mathcal{J} & -\mathcal{J} &\cdots & 0 & 0 & 0 & -\mathcal{J} & L_{t+4} & \cdots & -\mathcal{J} \\
 \vdots & \vdots & \vdots & \vdots & \ddots & \vdots & \vdots & \vdots & \vdots & \vdots & \ddots & \vdots\\
   -\mathcal{J} & -\mathcal{J} & -\mathcal{J} & 0 &\cdots & 0 & 0 & 0 & -\mathcal{J} & -\mathcal{J} & \cdots & L_{2(k+1)} \\
\end{block}
 \end{blockarray}_{2^kq},\]

 \[ L_{22} = 
 \begin{bmatrix}
       3 & -1 & 0 & 0 &\cdots & 0 & 0 & 0 & \cdots & 0 \\
   -1  & 3 & 0 & 0 &\cdots & 0 & 0 & 0 & \cdots & 0 \\
   0 & 0 & 3 & -1 & \cdots & 0 & 0 & 0 & \cdots & 0\\
 0 & 0 & -1 & 3 & \cdots & 0 & 0 & 0 & \cdots & 0 \\
 \vdots & \vdots &\vdots&\vdots & \ddots & \vdots & \vdots & \vdots & \ddots & \vdots \\
  0 & 0 & 0 & 0 & \cdots & 3 & -1 & 0 & \cdots & 0\\
 0 & 0 & 0 & 0 & \cdots & -1 & 3 & 0 & \cdots & 0 \\
    0 & 0 & 0 & 0 & \cdots & 0 & 0 & 1 & \cdots & 0\\
 \vdots & \vdots& \vdots & \vdots & \ddots & \vdots & \vdots & \vdots & \ddots & \vdots \\
   0 & 0 & 0 & 0 & \cdots & 0 & 0 & 0 & \cdots & 1 \\
\end{bmatrix}_{2^{k}q},\]

$$L_{3} = [l_{ij}]_{\phi(n)} = \begin{cases}
    e_3 , & \text{if $i = j$}\\
    -1, & \text{otherwise}
    \end{cases}, \ L_4 = [l_{ij}]_{2^{k-2}(q-1)} = \begin{cases}
    e_4, & \text{if $i=j$}\\
    -1, & \text{otherwise}
    \end{cases}, \ldots,$$
$$L_{t} = [l_{ij}]_{2(q-1)} = \begin{cases}
    e_t , & \text{if $i = j$}\\
    -1, & \text{otherwise}
    \end{cases}, \ L_{t+1} = [l_{ij}]_{(q-1)} = \begin{cases}
    e_{t+1}, & \text{if $i=j$}\\
    -1, & \text{otherwise}
    \end{cases},$$
$$L_{t+2} = [l_{ij}]_{q-1} = \begin{cases}
    e_{t+2} , & \text{if $i = j$}\\
    -1, & \text{otherwise}
    \end{cases}, \ L_{t+3} = [l_{ij}]_{2} = \begin{cases}
    e_{t+3}, & \text{if $i=j$}\\
    -1, & \text{otherwise}
    \end{cases},$$
$$L_{t+4} = [l_{ij}]_{4} = \begin{cases}
    e_{t+4}, & \text{if $i=j$}\\
    -1, & \text{otherwise}
    \end{cases}, \ldots, \ L_{2(k+1)} = [l_{ij}]_{2^{k-1}} = \begin{cases}
    e_{2(k+1)}, & \text{if $i=j$}\\
    -1, & \text{otherwise}
    \end{cases}.$$
    
Here $e_i$, for $1\leq i \leq 2(k+1)$, denote the degree of respective vertices in the graph $P(\mathcal{G})$, which is given precisely as $ e_1 = 2^{k+1}q-1,  e_2 = (3 \cdot 2^{k-1}-1)q, e_3 = 2^kq-1, e_4 = 2^{k-1}(2q-1)-1, \ldots, e_t = 2^k(q-1)+3, e_{t+1} = 2^{k}(q-1)+1$, $e_{t+2} = 2^{k}(q-1), e_{t+3} = 2^kq-2q+1,  e_{t+4} = 2^{k}q-4q+3, \ldots, e_{2(k+1)} = 2^{k-1}(q+1)-1.$

Now, considering the equitable partition $\pi$, given in Theorem \ref{t4.1} and using Proposition \ref{p3.1} and Equation (\ref{l2}) of Theorem \ref{t2.1}, we get $$ Char(L,x) = Char(\mathcal{Q}_1,x)  \prod_{i=1}^{2k+4} \frac{Char(H_i, x)}{(x-deg(v_i)+r_i)},$$ where $H_i$, for $1 \leq i \leq 2k+2$, $H_{2k+3}$, and $H_{2k+4}$ are the same as in Theorem \ref{t4.1} and $\mathcal{Q}_1$ denote the quotient matrix corresponding to the partition $\pi$ defined above. Hence,
\newline
$Char(L,x) = (x-1)^{2^{k-1}q}(x-2)^{2^{k-2}q-1}(x-1)^{2^{k-2}q}(x-2^{k}q)^{2^{k-1}(q-1)-1}(x-2^{k-1}(2q-1))^{2^{k-2}(q-1)-1}(x-2^{k}(q-1)-4)^{2q-3}(x-2^{k}(q-1)-2)^{q-1}(x-2^{k}(q-1)-1)^{q-2}(x-2^{k}q+2q-2)(x-2^{k}q-4q-4)^3(x-2^{k-1}(q+1))^{2^{k-1}-1}Char(\mathcal{Q}_1,x). \hspace{4.5cm} \hfill \Box$

\begin{cor}\label{c4.2}
    For $k=2$ and odd primes $q$, $Char(L,x)$ is\\
    $x(x-8q)(x-1)^{2q-1}(x-2)^q(x-4)^q(x-4q)^{2q-3}(x-4q+2)^{q-1}(x-4q+3)^{q-2}\big(-24q+64q^2+168q^3-304q^4+(4-20q-188q^2+276q^3+160q^4)x+(2+60q-74q^2-192q^3)x^2-(6+6q+82q^2)x^3-15qx^4+x^5\big).
 $
\end{cor}

\noindent{\textbf{Proof.}} By Theorem \ref{t4.2}, for $k=2$, we get 
$Char(L,x) =x(x-8q)(x-1)^{2q-1}(x-2)^q(x-4)^q(x-4q)^{2q-3}(x-4q+2)^{q-2}(x-4q+3)^{q-2}Char(\mathcal{Q}_1,x),$ where 
$$ \mathcal{Q}_1=\begin{bmatrix}
    8q-1 & -1 & -2(q-1) & 1-q & 1-q & -2 & -2q & -2q\\
    -1 & 5q & -2(q-1) & 1-q & 0 & -2 & -2q & 0\\
    -1 & -1 & 2q+2 & 1-q & 1-q & -2 & 0 & 0\\
    -1 & -1 & -2(q-1) & 3q-1 & 1-q & 0 & 0 & 0\\
    -1 & 0 & -2(q-1) & 1-q & 3q-2 & 0 & 0 & 0\\
    -1 & -1 & -2(q-1) & 0 & 0 & 2p & 0 & 0\\
     -1 & -1 & 0 & 0 & 0 & 0 & 2 & 0\\
     -1 & 0 & 0 & 0 & 0 & 0 & 0 & 1\\
\end{bmatrix}.$$

The characteristic polynomial of the quotient matrix $\mathcal{Q}_1$ is\\
$|xI-\mathcal{Q}_1| = x(x-1)(x-8q) \big(-24q+64q^2+168q^3-304q^4+(4-20q-188q^2+276q^3+160q^4)x+(2+60q-74q^2-192q^3)x^2-(6+6q+82q^2)x^3-15qx^4+x^5\big).$

Now using the value of $|xI-\mathcal{Q}_1|$ in    $Char(L,x)$ , we get the required result.
\hfill$\Box$

\begin{thm}\label{t4.3}
The characteristic polynomial corresponding to the signless Laplacian matrix $Q$ of $P(\mathcal{G})$ is
\newline
    $Char(Q,x)=(x-1)^{2^{k-1}q-1}(x-2)^{2^{k-2}q}(x-4)^{2^{k-2}q-1}(x-2^{k}q+2)^{2^{k-1}(q-1)-1}(x-2^{k-1}(2q-1)+2)^{2^{k-2}(q-1)-1}\cdots(x-2^{k}(q-1)-2)^{2q-3}(x-2^{k}(q-1)^{q-2}(x-2^{k}(q-1)-1)^{q-2}(x-2^{k}q+2q)(x-2^{k}q-4q+4)^3\cdots(x-2^{k-1}(q+1)^{2^{k-1}-1}Char(\mathcal{Q}_2,x),$ where $\mathcal{Q}_2$ is the quotient matrix given below
$$\begin{bmatrix}
 \begin{array}{c|cccccccccc|c}
   \multirow{2}{*}{$M_1$} & \phi(2^kq) & \phi(2^{k-1}q) &\cdots & \phi(2^2q) & \phi(2q) & \phi(q)  & 2  & 2^2  & \cdots & \phi(2^k) & \multirow{2}{*}{${A}_2$} \\
   & \phi(2^kq)  & \phi(2^{k-1}q)  &\cdots & \phi(2^2q) & \phi(2q) & 0 & 2  & 2^2  & \cdots & \phi(2^k) &  \\ \hline
   \multirow{12}{*}{$B^T$}  & q_3 & \phi(2^{k-1}q)  &\cdots & \phi(2^2q) & \phi(2q) & \phi(q) & 2 & 2^2 & \cdots & \phi(2^k) & \multirow{12}{*}{$O^T$} \\
   & \phi(2^kq) &q_4 &\cdots & \phi(2^2q) & \phi(2q) & \phi(q) & 2 & 2^2 & \cdots & 0 & \\
  & \vdots & \vdots & \ddots & \vdots & \vdots & \vdots & \vdots & \vdots & \ddots & \vdots & \\
  & \phi(2^{k}q) & \phi(2^{k-1}q) &\cdots & q_t & \phi(2q) & \phi(q) & 2 & 0 & \cdots & 0 &  \\
  & \phi(2^{k}q) & \phi(2^{k-1}q) &\cdots & \phi(2^2q) & q_{t+1} & \phi(q) & 0 & 0 & \cdots & 0 &\\
   & \phi(2^{k}q) & \phi(2^{k-1}q) &\cdots & \phi(2^2q) & \phi(2q) & q_{t+2} & 0 & 0 & \cdots & 0 & \\
 & \phi(2^{k}q) & \phi(2^{k-1}q) &\cdots & \phi(2^2q) & 0 & 0 & q_{t+3} & 2^2 & \cdots & \phi(2^k) & \\
 & \phi(2^{k}q) & \phi(2^{k-1}q) &\cdots & 0 & 0 & 0 & 2 & q_{t+4} & \cdots & \phi(2^k) &  \\
  & \vdots & \vdots & \ddots & \vdots & \vdots & \vdots & \vdots & \vdots & \ddots & \vdots & \\
    % 2^{k-2} &  J & J & J & J &\cdots & 0 & 0 & 0 & J & J & \cdots & J \\
 & \phi(2^{k}q) & 0 &\cdots & 0 & 0 & 0 & 2 & 2^2 & \cdots & q_{2(k+1)} & \\ \hline
  \multirow{2}{*}{$A_3$} & 0 & 0 &\cdots & 0 & 0 & 0 & 0 & 0 & \cdots & 0 & \multirow{2}{*}{$M_2$}\\
   & 0 & 0 &\cdots & 0 & 0 & 0 & 0 & 0 & \cdots & 0 &  \\
\end{array}
 \end{bmatrix}.$$
 
where $M_1=\begin{bmatrix}
    q_1 & 1\\
    1 & q_2
\end{bmatrix}, M_2=\begin{bmatrix}
    4 & 0\\
    0 & 1
\end{bmatrix},$ $q_1=2^{k+1}q-1$, $q_2=(3\cdot2^{k-1}-1)q,$ $q_3=3\cdot2^{k-1}q-2,$ $q_4=2^{k-2}(5q-3)-2,$ $q_t=2^{k}(q-1)+2q,$ $q_{t+1}=2^{k}(q-1)+q-1,$ $q_{t+2}=2^{k}(q-1)+q-2,$ $q_{t+3}=2^{k}q-2q+2,$ $q_{t+4}=2^{k}q-4q+6,$ $q_{2{k+1}}=2^{k-1}(q+2)-2.$
\end{thm}

\noindent{\textbf{Proof.}} By the same indexing of the vertices as in Theorem \ref{t4.1}, the signless Laplacian matrix of $P(\mathcal{G})$ is
$$Q = \begin{bmatrix}
    Q_{11} & Q_{12}\\
    Q^t_{12} & Q_{22}
\end{bmatrix}, \text{ where }$$  $$ Q_{12} = [q_{ij}]_{2^kq} = \begin{cases}
    1, & \text{if $i = 1$ and $2^{k}q+1 \leq j \leq 2^{k+1}q$} \\
    1, & \text{if $i=2$ and $2^{k}q+1 \leq j \leq 3\cdot2^{k-1}q+1$}\\
    0, & \text{otherwise}
    \end{cases}$$
    
\[ \hspace{-1cm} Q_{11} = 
 \begin{blockarray}{cccccccccccc}
   % 1 & 1 & {2^{k-1}(q-1)} & {2^{k-2}(q-1)} & \dots & {2(q-1)} & {q-1} & {q-1} & 2 & 2^2 & \dots & {2^{k-1}}\\
 \begin{block}{[cccccccccccc]}
   e_1 & 1 & \mathcal{J} & \mathcal{J} &\cdots & \mathcal{J} & \mathcal{J} & \mathcal{J} & \mathcal{J} & \mathcal{J} & \cdots & \mathcal{J} \bigstrut[t] \\
  1 & e_2 &  \mathcal{J} & \mathcal{J} &\cdots & \mathcal{J} & \mathcal{J} & 0 & \mathcal{J} &  \mathcal{J} & \cdots &  \mathcal{J} \\
    \mathcal{J} & \mathcal{J} & Q_3 &  \mathcal{J} & \cdots & \mathcal{J} & \mathcal{J} & \mathcal{J} &  \mathcal{J} & \mathcal{J} & \cdots &  \mathcal{J} \\
   \mathcal{J} &  \mathcal{J} &  \mathcal{J} & Q_4 &\cdots &  \mathcal{J} & \mathcal{J} &  \mathcal{J} &  \mathcal{J} & \mathcal{J} & \cdots & 0 \\
 \vdots & \vdots & \vdots & \vdots & \ddots & \vdots & \vdots & \vdots & \vdots & \vdots & \ddots & \vdots\\
  \mathcal{J} & \mathcal{J} & \mathcal{J} & \mathcal{J} &\cdots & Q_t & \mathcal{J} & \mathcal{J} & \mathcal{J} & 0 & \cdots & 0 \\
  \mathcal{J} & \mathcal{J} & \mathcal{J} & \mathcal{J} &\cdots & \mathcal{J} & Q_{t+1} & \mathcal{J} & 0 & 0 & \cdots & 0 \\
  \mathcal{J} & 0 &  \mathcal{J} & \mathcal{J} &\cdots & \mathcal{J} & \mathcal{J} & Q_{t+2} & 0 & 0 & \cdots & 0 \\
  \mathcal{J} & \mathcal{J} & \mathcal{J} & \mathcal{J} &\cdots & \mathcal{J} & 0 & 0 & Q_{t+3} & \mathcal{J} & \cdots & \mathcal{J} \\
 \mathcal{J} & \mathcal{J} & \mathcal{J} & \mathcal{J} &\cdots & 0 & 0 & 0 & \mathcal{J} & Q_{t+4} & \cdots & \mathcal{J} \\
 \vdots & \vdots & \vdots & \vdots & \ddots & \vdots & \vdots & \vdots & \vdots & \vdots & \ddots & \vdots\\
   \mathcal{J} & \mathcal{J} & \mathcal{J} & 0 &\cdots & 0 & 0 & 0 & \mathcal{J} & \mathcal{J} & \cdots & Q_{2(k+1)} \\
\end{block}
 \end{blockarray}_{2^kq},\]

 \[ Q_{22} = 
 \begin{bmatrix}
       3 & 1 &0&0 &\cdots & 0 & 0 & 0 & \cdots & 0 \\
   1  & 3 &0 & 0 &\cdots & 0 & 0 & 0 & \cdots & 0 \\
   0 & 0 &3& 1& \cdots & 0 & 0 & 0 & \cdots & 0\\
 0 & 0 & 1&3& \cdots & 0 & 0 & 0 & \cdots & 0 \\
 \vdots & \vdots &\vdots&\vdots & \ddots & \vdots & \vdots & \vdots & \ddots & \vdots \\
  0 & 0 &0&0& \cdots & 3 & 1 & 0 & \cdots & 0\\
 0 & 0 &0&0& \cdots & 1 & 3 & 0 & \cdots & 0 \\
    0 & 0 &0&0& \cdots & 0 & 0 & 1 & \cdots & 0\\
 \vdots & \vdots& \vdots & \vdots & \ddots & \vdots & \vdots & \vdots & \ddots & \vdots \\
   0 & 0 &0&0& \cdots & 0 & 0 & 0 & \cdots & 1 \\
\end{bmatrix}_{2^{k}q},\]

$$Q_{3} = [q_{ij}]_{\phi(n)} = \begin{cases}
    e_3 , & \text{if $i = j$}\\
    1, & \text{otherwise}
    \end{cases}, \ Q_4 = [q_{ij}]_{2^{k-2}(q-1)} = \begin{cases}
    e_4, & \text{if $i=j$}\\
    1, & \text{otherwise}
    \end{cases}, \ldots,$$
$$Q_{t} = [q_{ij}]_{2(q-1)} = \begin{cases}
    e_t , & \text{if $i = j$}\\
    1, & \text{otherwise}
    \end{cases}, \ Q_{t+1} = [q_{ij}]_{(q-1)} = \begin{cases}
    e_{t+1}, & \text{if $i=j$}\\
    1, & \text{otherwise}.
    \end{cases},$$
$$Q_{t+2} = [q_{ij}]_{q-1} = \begin{cases}
    e_{t+2} , & \text{if $i = j$}\\
    1, & \text{otherwise}
    \end{cases}, \ Q_{t+3} = [q_{ij}]_{2} = \begin{cases}
    e_{t+3}, & \text{if $i=j$}\\
    1, & \text{otherwise}
    \end{cases},$$
$$Q_{t+4} = [q_{ij}]_{4} = \begin{cases}
    e_{t+4}, & \text{if $i=j$}\\
    1, & \text{otherwise}
    \end{cases}, \ldots, \ Q_{2(k+1)} = [q_{ij}]_{2^{k-1}} = \begin{cases}
    e_{2(k+1)}, & \text{if $i=j$}\\
    1, & \text{otherwise}
    \end{cases}.$$
Here $e_i$, for $1\leq i \leq 2(k+1)$, represents the degrees of the corresponding vertices in $P(\mathcal{G})$, as described in Theorem~\ref{t4.1}.
Now, considering the equitable partition $\pi$, given in Theorem \ref{t4.1} and using Proposition \ref{p3.1} and Equation (\ref{l2}) of Theorem \ref{t2.1}, we get $$ Char(L,x) = Char(\mathcal{Q}_2,x)  \prod_{i=1}^{2k+4} \frac{Char(H_i, x)}{(x-deg(v_i)-r_i)},$$ where $H_i$, for $1 \leq i \leq 2k+2$, $H_{2k+3}$, and $H_{2k+4}$ are the same as in Theorem \ref{t4.1} and $\mathcal{Q}_2$ be the quotient matrix corresponding to the partition $\pi$ introduced above. Consequently,
\newline
$Char(Q,x)=(x-1)^{2^{k-1}q-1}(x-2)^{2^{k-2}q}(x-4)^{2^{k-2}q-1}(x-2^{k}q+2)^{2^{k-1}(q-1)-1}(x-2^{k-1}(2q-1)+2)^{2^{k-2}(q-1)-1}\cdots(x-2^{k}(q-1)-2)^{2q-3}(x-2^{k}(q-1)^{q-2}(x-2^{k}(q-1)-1)^{q-2}(x-2^{k}q+2q)(x-2^{k}q_4q+4)^3\cdots(x-2^{k-1}(q+1)^{2^{k-1}-1}Char(\mathcal{Q}_2,x)$.$                       $\hfill$\Box$

\begin{cor}
In the case $k=2$ and $q\neq 2$ prime, the characteristic polynomial $Char(Q,x)$ is
\newline $(x-2q)(x-1)^{2q-1}(x-4)^{q-1}(x-2)^q(x-4q+2)^{2q-3}(x-4q-2)^{q-2}(x-4q+5)^{q-2}\big((25088q^{6}-64576q^{5}+39424q^{4}+19520q^{3}-27712q^{2}+9216q-960)-(41984q^{6}-67200q^{5}-25648q^{4}+84096q^{3}-36608q^{2}+2176q+688)x+(10240q^{6}+36736q^{5}-100016q^{4}+43656q^{3}+19192q^{2}-12072q+1176)x^{2}-(14848q^{5}+3920q^{4}-40676q^{3}+22796q^{2}-108q-860)x^{3}+(8480q^{4}-4340q^{3}-6778q^{2}+3696q-194)x^{4}-(2464q^{3}-1588q^{2}-454q+196)x^{5}+(386q^{2}-203q-8)x^{6}-(31q-9)x^{7}+x^{8} \big).$    
\end{cor}

\noindent{\textbf{Proof.}}  By Theorem \ref{t4.3}, for $k=2$, we get 
$Char(Q,x) =(x-2q)(x-1)^{2q-1}(x-4)^{q-1}(x-2)^q(x-4q+2)^{2q-3}(x-4q-2)^{q-2}(x-4q+5)^{q-2}Char(\mathcal{Q}_2,x),$ where

$$ \mathcal{Q}_2=\begin{bmatrix}
    8q-1 & 1 & 2(q-1) & q-1 & q-1 & 2 & 2q & 2q \\
    1 & 5q & 2(q-1) & q-1 & 0 & 2 & 2q & 0 \\
    1 & 1 & 6q-4 & q-1 & q-1 & 2 & 0 & 0 \\
    1 & 1 & 2(q-1) & 5q-5 & q-1 & 0 & 0 & 0 \\
    1 & 0 & 2(q-1) & q-1 & 5q-6 & 0 & 0 & 0 \\
    1 & 1 & 2(q-1) & 0 & 0 & 2q+2 & 0 & 0 \\
    1 & 1 & 0 & 0 & 0 & 0 & 4 & 0 \\
    1 & 0 & 0 & 0 & 0 & 0 & 0 & 1 \\
\end{bmatrix}.$$

Note that the characteristic polynomial of the quotient matrix $\mathcal{Q}_2$ is
\newline
$char(\mathcal{Q}_2,x)=\big((25088q^{6}-64576q^{5}+39424q^{4}+19520q^{3}-27712q^{2}+9216q-960)-(41984q^{6}-67200q^{5}-25648q^{4}+84096q^{3}-36608q^{2}+2176q+688)x+(10240q^{6}+36736q^{5}-100016q^{4}+43656q^{3}+19192q^{2}-12072q+1176)x^{2}-(14848q^{5}+3920q^{4}-40676q^{3}+22796q^{2}-108q-860)x^{3}+(8480q^{4}-4340q^{3}-6778q^{2}+3696q-194)x^{4}-(2464q^{3}-1588q^{2}-454q+196)x^{5}+(386q^{2}-203q-8)x^{6}-(31q-9)x^{7}+x^{8} \big).$

Substituting the value of $|xI-\mathcal{Q}_2|$ into $Char(Q,x)$ we get the required result. This completes the proof.
\hfill$\Box$

\section{Spectral Radius Bounds for $P(\mathcal{G})$} \label{Sec5}
Let us recall the following result to determine bounds for the spectral radius of the matrices associated with the $P(\mathcal{G})$. 
\begin{thm}[Brouwer and Haemers \cite{bh(2012)}] \label{t2.3}
	Let $\mathcal{A}$ and $\mathcal{B}$ be two symmetric matrices of order $n$. Then for the largest eigenvalue $\lambda_{max}$ we have

$$ \lambda_{max}(\mathcal{A}+\mathcal{B}) \leq  \lambda_{max}(\mathcal{A}) + \lambda_{max}(\mathcal{B}).$$ 
\end{thm}

Let $A(P(\mathbb{Z}_n))$ denote the adjacency matrix of the power graph of the finite cyclic group $\mathbb{Z}_n$. Then we show the following.

\begin{thm}\label{5.1}
    The largest eigenvalue of $A(P(\mathbb{Z}_n))$ is bounded above by $\frac{n-3+\sqrt{n^2-2n+4l+1}}{2}$ and equality holds if and only if $P(\mathbb{Z}_n)$ is complete.
\end{thm}

\noindent{\textbf{Proof.}} Define $T_1=\{a\in \mathbb{Z}_n : \gcd(a,n)=1\}\cup\{0\}$
and $T_2=\{a\in \mathbb{Z}_n : \gcd(a,n)\neq 1\}.$ Then $\mathbb{Z}_n=T_1\cup T_2$, and
$|T_1|=\phi(n)+1= l(\text{say}).$ Thus the adjacency matrix of $P(\mathbb{Z}_n)$ is given by 
\begin{equation}\label{1}
A=A(P(\mathbb{Z}_n))=\begin{bmatrix}
    (J-I)_{l \times l} & J_{l \times (n-l)} \\
    J^t_{l \times (n-l)} & B_{(n-l) \times (n-l)}
\end{bmatrix}.\end{equation}
Let $\lambda_1$ denote the largest eigenvalue of $A$. By the Perron-Frobenius theorem \cite{bh(2012)}, there exists a positive eigenvector $Y^t=(y_1,y_2,\ldots,y_n)$, where $y_i>0$ for all $1\leq i\leq n$, corresponding to $\lambda_1$. Hence,
\begin{equation}\label{2}
    AY=\lambda_1 Y.
\end{equation}
Let $$\ y_{r} = \max _{1\leq i \leq l}(y_i) \ \text{and} \ y_s = \max_{l < i \leq n}(y_i).$$ 
The $r^{th}$ equation of Equation (\ref{2}) is
\begin{equation}\label{3}
   a_{r1}y_1+a_{r2}y_2+ \cdots + a_{rl}y_l+a_{r(l+1)}y_{l+1}+ \cdots+a_{rn}y_n = \lambda_1y_r.
\end{equation}
From $(\ref{1})$, we have 
\begin{equation}\label{4}
a_{rr}=0,\quad \text{while} \quad a_{ri}=1 \ \text{for all} \ i\neq r,\; 1\leq i\leq n.
\end{equation}
Using the above values of $a_{rr}$ and $a_{ri}$ in Equation $(\ref{3})$, we get
\begin{align*}
y_1+ \cdots + y_{r-1}+y_{r+1}+\cdots+ y_l+y_{l+1}+ \cdots+y_n = \lambda_1y_r \\
\Rightarrow  y_r+ \cdots + y_{r}+y_{r}+\cdots+ y_r+y_{s}+ \cdots+y_s \geq \lambda_1y_r \\
\Rightarrow   (l-1)y_r+(n-l)y_s \geq \lambda_1 y_r
\end{align*} 
\begin{equation}\label{5}
 \ \ \ \ \ \ \ \ \ \ \ \ \ \ \ \ \ \ \ \ \ \ \ \ \ \ \ \ \ \ \ \ \ \ \ \ \ \ \  \Rightarrow    (\lambda_1-l+1)y_r \leq (n-l)y_s.
\end{equation}
The $s^{th}$ equation of Equation (\ref{2}) is 
\begin{equation}\label{6}
a_{s1}y_1+a_{s2}y_2+ \cdots + a_{sl}y_l+a_{s(l+1)}y_{l+1}+ \cdots+a_{sn}y_n = \lambda_1y_s.
\end{equation}
By (\ref{1}), we have
\begin{equation}\label{7}
    a_{si} = 1 \ \forall \ 1 \leq i \leq l, a_{si} = 0  \ \text{for at least one} \ i \in \{l+1,l+2,\ldots,n\}, \text{ and } a_{ss} =0.
  \end{equation}
By the fact that the row sum of an adjacency matrix gives the degree of vertex associated to that row, we get 
\begin{equation}\label{8}
    a_{s(l+1)}+ \cdots+a_{sn} \leq n-l-2.
\end{equation}
Using Equations (\ref{7}) and (\ref{8}) in Equation (\ref{6}), we get
$$ \lambda_1y_s \leq ly_r + (n-l-2)y_s$$
\begin{equation}\label{9}
 \Rightarrow (\lambda_1-n+l+2)y_s \leq ly_r. 
\end{equation}
From (\ref{5}) and (\ref{9}), we get
$$(\lambda_1-l+1)(\lambda_1-n+l+2)y_s \leq l(n-l)y_s.$$
Since $y_s > 0$, we get 
$$(\lambda_1-l+1)(\lambda_1-n+l+2) \leq l(n-l)$$ 
$$\Rightarrow \lambda_1^2-(n-3)\lambda_1-(n+l-2) \leq 0$$
 \begin{equation}\label{10}
 \Rightarrow   \lambda_1 \leq \frac{n-3+\sqrt{n^2-2n+4l+1}}{2}.
\end{equation}
The inequality (\ref{10}) turns into an equality if $a_{si} = 1$, for $l+1 \leq i \leq n, \text{ and } a_{ss} = 0$. Then every non-generator element of 
$\mathbb{Z}_n$ becomes a generator of $\mathbb{Z}_n$, which is possible only if $n$ is some prime power, i.e., $P(\mathbb{Z}_n)$ is complete. Consequently, taking $l=n$ in (\ref{10}), we get $\lambda_1 = n-1$. \hfill $\Box$

\begin{thm}\label{5.2}
The spectral radius of $P(\mathbb{Z}_n)$ satisfies
\[
\lambda_1\big(P(\mathbb{Z}_n)\big)\geq
\frac{\,l-1+\sqrt{4nl-3l^2-2l+1}\,}{2}.
\]
Moreover, equality holds precisely when $P(\mathbb{Z}_n)$ is a complete graph.
\end{thm}
\noindent{\textbf{Proof.}} From (\ref{1}) and Equation~(\ref{2}) of Theorem~\ref{5.1}, the structure of $A(P(\mathbb{Z}_n))$ and the eigenvector associated with its largest eigenvalue are known. Let
$$\ y_{r} = \min _{1\leq i \leq l}(y_i) \ \text{and} \ y_s = \min_{l < i \leq n}(y_i).$$
From Equations (\ref{3}) and (\ref{4}), we get
\begin{equation}\label{11}
   (n-l)y_s \leq (\lambda_1-l+1)y_r.
\end{equation}
Also, by (\ref{1}), we get
\begin{equation}\label{12}
    a_{si}=1 \ \text{for all} \ 1 \leq i \leq l \ \text{and} \ a_{ss}=0.
\end{equation}
As the $s^{\text{th}}$ vertex of $P(\mathbb{Z}_n)$ corresponds to a non-generator of $\mathbb{Z}_n$, it may be adjacent only to the vertices belonging to $T_1$. Therefore,
\begin{equation}\label{13}
    a_{si}=0 \ \text{for all} \ i, \ \text{where} \ l+1 < i \leq n.
    \end{equation}
Using (\ref{12}) and (\ref{13}) in Equation (\ref{6}), we get
\begin{equation}\label{14}
    \lambda_1y_s \geq l y_r.
\end{equation}
From (\ref{11}) and (\ref{14}), we have
$$l(n-l)y_s \leq (\lambda_1-l+1)\lambda_1y_s.$$
Since $y_s >0$, we have
$$l(n-l) \leq (\lambda_1-l+1)\lambda_1$$
\begin{equation}\label{15}
\Rightarrow  \lambda_1 \geq \frac{l-1+\sqrt{4nl-3l^2-2l+1}}{2}.
\end{equation}
Equality holds in (\ref{15}) only when $a_{si} = 1$, for $l+1 \leq i \leq n$. Then every non-generator element of $\mathbb{Z}_n$ becomes a generator of $\mathbb{Z}_n$, which means $n$ is some prime power, and hence $P(\mathbb{Z}_n)$ is complete.  Now, taking $l = n$, we get $\lambda_1 = n-1$. \hfill $\Box$

Now, using Theorem \ref{5.1} and Theorem \ref{5.2}, we determine bounds for the spectral radius of the adjacency matrix of  $P(\mathcal{G})$.

\begin{thm}\label{5.3}
The spectral radius of $A(P(\mathcal{G}))$ is given by 
$$\frac{2^{k-1}(q-1)+n_1}{2} \leq \lambda_1(A(P(\mathcal{G}))) \leq 2^{k-1}q-1+\sqrt{2^kq}+\frac{\sqrt{(2^kq)^2-2^{k+1}+1}+\sqrt{1+2^{k+1}q}}{2},$$ where $k\geq2$ and $q\neq2$ is prime and $n_1 = \sqrt{2^{2k+1}q(q-1)-3\cdot2^{2k-2}(q-1)^2+2^{k+2}-4}$.
\end{thm}

\noindent{\textbf{Proof.}}
To define the adjacency matrix $A$, we arrange its rows and columns according to the following ordering of the elements of $\mathcal{G}$: first the identity element, followed by $r^{2^{k-1}q}$, then the elements generated by $r$, namely $r,r^2,\ldots,r^{2^kq-1},$ then ${sr,sr^{2^{k-1}q+1},sr^{3q},sr^{2^{k-1}q+3}, \ldots, sr^{2^{k-1}q-1},sr^{2^kq-1}}$, and finally $s,sr^2,\ldots, sr^{2^kq-2}$. With this indexing, the adjacency matrix $A$ takes the following form.
\begin{equation}\label{16}
    A(P(\mathcal{G}))=\begin{bmatrix}
        A(P(\mathbb{Z}_{2^kq})) & C_{2^kq} \\
        C_{2^kq}^t & B_{2^kq} 
    \end{bmatrix} = \begin{bmatrix}
        A(P(\mathbb{Z}_{2^kq})) & C_1\\
        C^t_1 & O_{2^{k}q}
    \end{bmatrix}+\begin{bmatrix}
        O_{2^{k}q} & C_2\\
        C_2^t & B_{2^kq}
    \end{bmatrix},\end{equation} where  $$C_1 = \begin{bmatrix}
        1 & 1 & \cdots & 1 & 1 & 1 & \cdots & 1 \\
        0 & 0 & \cdots & 0 & 0 & 0 & \cdots & 0 \\
        \vdots & \vdots & \ddots & \vdots & \vdots & \vdots & \ddots & \vdots\\
        0 & 0 & \cdots & 0 & 0 & 0 & \cdots & 0 \\
        0 & 0 & \cdots & 0 & 0 & 0 & \cdots & 0 \\
        0 & 0 & \cdots & 0 & 0 & 0 & \cdots & 0 \\
     \vdots & \vdots & \ddots & \vdots & \vdots & \vdots & \ddots & \vdots\\
     0 & 0 & \cdots & 0 & 0 & 0 & \cdots & 0 \\
    \end{bmatrix}_{2^kq}, C_2= \begin{bmatrix}
        0 & 0 & \cdots & 0 & 0 & 0 & \cdots & 0 \\
        1 & 1 & \cdots & 1 & 1 & 0 & \cdots & 0 \\
        0 & 0 & \cdots & 0 & 0 & 0 & \cdots & 0 \\
        \vdots & \vdots & \ddots & \vdots & \vdots & \vdots & \ddots & \vdots\\
        0 & 0 & \cdots & 0 & 0 & 0 & \cdots & 0 \\
        0 & 0 & \cdots & 0 & 0 & 0 & \cdots & 0 \\
        0 & 0 & \cdots & 0 & 0 & 0 & \cdots & 0 \\
        \vdots & \vdots & \ddots & \vdots & \vdots & \vdots & \ddots & \vdots\\
        0 & 0 & \cdots & 0 & 0 & 0 & \cdots & 0 \\
    \end{bmatrix}_{2^kq},$$ and $$B_{2^kq}= \begin{bmatrix}
        0 & 1 & \cdots & 0 & 0 & 0 & \cdots & 0\\
        1 & 0 & \cdots & 0 & 0 & 0 & \cdots & 0\\
        \vdots & \vdots & \ddots & \vdots & \vdots & \vdots & \ddots & \vdots\\
        0 & 0 & \cdots & 0 & 1 & 0 & \cdots & 0\\
        0 & 0 & \cdots & 1 & 0 & 0 & \cdots & 0 \\
        0 & 0 & \cdots & 0 & 0 & 0 & \cdots &  0\\
        \vdots & \vdots & \ddots & \vdots & \vdots & \vdots & \ddots & \vdots\\
        0  & 0 & \cdots & 0 & 0 & 0 &  \cdots & 0
    \end{bmatrix}.$$

\noindent By Theorem \ref{t2.3} and (\ref{16}), we have
\begin{equation}\label{17}
  \lambda_1(A(P(\mathcal{G}))) \leq \lambda_1\bigg(\begin{bmatrix}
A(P(\mathbb{Z}_{2^kq})) & C_1 \\
C^t_1 & O
\end{bmatrix}\bigg) +\lambda_1\bigg(\begin{bmatrix}
    O & C_2 \\
    C^t_2 & B_{2^kq}
\end{bmatrix}\bigg).  
\end{equation}
Let \begin{equation}\label{18}
    \begin{bmatrix}
A(P(\mathbb{Z}_{2^kq})) & C_1 \\
C^t_1 & O
\end{bmatrix} = \begin{bmatrix}
  A(P(\mathbb{Z}_{2^kq})) & O \\
  O & O
\end{bmatrix} + \begin{bmatrix}
    O & C_1\\
    C^t_1 & O
\end{bmatrix}.\end{equation}
 From Theorem \ref{t2.3} and (\ref{18}), we get
\begin{equation}\label{19}
    \lambda_1\bigg(\begin{bmatrix}
A(P(\mathbb{Z}_{2^kq})) & C_1 \\
C^t_1 & O
\end{bmatrix}\bigg) \leq \lambda_1\bigg(\begin{bmatrix}
A(P(\mathbb{Z}_{2^kq})) & O \\
O & O
\end{bmatrix}\bigg) +\lambda_1\bigg(\begin{bmatrix}
    O & C_1 \\
    C^t_1 & O
\end{bmatrix}\bigg).  
\end{equation}
The characteristic polynomial of $\begin{bmatrix}
    O & C_1 \\
    C^t_1 & O
\end{bmatrix}$ is $q_1(x) = \begin{vmatrix}
    xI & -C_1 \\
 - C^t_1 & xI
\end{vmatrix}.$

 \noindent Applying $R_1 \rightarrow xR_1$ and then $R_1 \rightarrow R_1+R_{2^kq+1}+R_{2^kq+2}+\cdots+R_{2^{k+1}q}$, we get
 
 $$ = \frac{1}{x} \begin{blockarray}{cccc|cccccc}
 \begin{block}{|cccc|cccccc|}
    x^2-2^kq & 0 & \cdots & 0 & 0 & \cdots & 0 & 0 & \cdots & 0\\
    0 & x & \cdots & 0 & 0 & \cdots & 0 & 0 & \cdots & 0 \\
    \vdots & \vdots & \ddots & \vdots & \vdots & \ddots & \vdots & \vdots & \ddots & \vdots\\
    0 & 0 & \cdots & x & 0 & \cdots & 0 & 0 & \cdots & 0 \\ \hline
    -1 & 0 & \cdots & 0 & x & \cdots & 0 & 0 & \cdots & 0\\
     \vdots & \vdots & \ddots & \vdots & \vdots & \ddots & \vdots & \vdots & \ddots & \vdots\\
     -1 & 0 & \cdots & 0 & 0 & \cdots & x & 0 & \cdots & 0\\
     -1 & 0 & \cdots & 0 & 0 & \cdots & 0 & x & \cdots & 0\\
     \vdots & \vdots & \ddots & \vdots & \vdots & \ddots & \vdots & \vdots & \ddots & \vdots\\
     -1 & 0 & \cdots & 0 & 0 & \cdots & 0 & 0 & \cdots & x\\
\end{block}
 \end{blockarray} \, .\vspace*{-2.5\baselineskip}$$

\vspace{.7cm}

\noindent This gives $q_1(x)= x^{2^{k+1}q-2}( x^2-2^{k}q)$.

\noindent Thus the largest eigenvalue of the matrix $\begin{bmatrix}
    O & C_1 \\
    C^t_1 & O
\end{bmatrix}$ is \begin{equation}\label{20}
    \lambda_1\bigg(\begin{bmatrix}
    O & C_1 \\
    C^t_1 & O
\end{bmatrix}\bigg) = \sqrt{2^kq}.
    \end{equation} 

\noindent The characteristic polynomial of $\begin{bmatrix}
    O & C_2 \\
    C^t_2 & B_{2^kq}
\end{bmatrix}$ is $q_2(x)= \begin{vmatrix}
    xI & -C_2 \\
    -C^t_2 & xI - B_{2^kq}
\end{vmatrix}.$

\noindent Applying $R_2 \rightarrow (x-1)R_2$ and $R_2 \rightarrow R_2 + R_{2^{k}q+1}+R_{2^{k}q+2}+\cdots+R_{3\cdot2^{k-1}q},$ we get
 
$$=\frac{1}{x-1} \begin{blockarray}{cccc|cccccccc}
 \begin{block}{|cccc|cccccccc|}
    x & 0 & \cdots & 0 & 0 & 0 & \cdots & 0 & 0 & 0 & \cdots & 0 \\
    0 & x(x-1)-2^{k-1}q & \cdots & 0 & 0 & 0 & \cdots & 0 & 0 & 0 & \cdots & 0 \\
    \vdots & \vdots & \ddots & \vdots & \vdots & \vdots & \ddots & \vdots & \vdots & \vdots & \ddots & \vdots \\
    0 & 0 & \cdots & x & 0 & 0 & \cdots & 0 & 0 & 0 & \cdots & 0 \\ \hline
    0 & -1 & \cdots & 0 & x & -1 & \cdots & 0 & 0 & 0 & \cdots & 0 \\
    0 & -1 & \cdots & 0 & -1 & x & \cdots & 0 & 0 & 0 & \cdots & 0 \\
     \vdots & \vdots & \ddots & \vdots & \vdots & \vdots & \ddots & \vdots & \vdots & \vdots & \ddots & \vdots \\
    0 & -1 & \cdots & 0 & 0 & 0 & \cdots & x & -1 & 0 & \cdots & 0 \\ 
    0 & -1 & \cdots & 0 & 0 & 0 & \cdots & -1 & x & 0 & \cdots & 0 \\ 
     0 & 0 & \cdots & 0 & 0 & 0 & \cdots & 0 & 0 & x & \cdots & 0 \\
     \vdots & \vdots & \ddots & \vdots & \vdots & \vdots & \ddots & \vdots & \vdots & \vdots & \ddots & \vdots \\
     0 & 0 & \cdots & 0 & 0 & 0 & \cdots & 0 & 0 & 0 & \cdots & x \\
\end{block}
 \end{blockarray}\vspace*{-2.5\baselineskip},$$
 
\vspace{.7cm}
This gives $q_2(x)= x^{3.2^{k-1}q-1}(x-1)^{2^{k-2}q-1}(x+1)^{2^{k-2}q}(x^2-x-2^{k-1}q)$.

Hence, the spectral radius of the matrix
 $\begin{bmatrix}
    O & C_2 \\
    C^t_2 & B_{2^{k}q}
\end{bmatrix}$ is 
\begin{equation}\label{21}
 \lambda_1\bigg(\begin{bmatrix}
    O & C_2 \\
    C^t_2 & B_{2^{k}q}
\end{bmatrix}\bigg)  = \frac{1+\sqrt{1+2^{k+1}q}}{2}.
\end{equation}

Thus, from (\ref{17}), (\ref{19}), (\ref{20}), (\ref{21}), and Theorem \ref{5.1}, we have
\begin{equation}\label{22}
\lambda_1(A(P(\mathcal{G}))) \leq 2^{k-1}q-1+\sqrt{2^{k}q}+\frac{\sqrt{1+2^{k+1}q}+\sqrt{(2^{k}q)^2}-2^{k+1}+1}{2}.
\end{equation}

Furthermore, Theorem \ref{5.2} together with \cite[Theorem 7.8.13]{crs(2011)} gives
\begin{equation}\label{23}
    \lambda_1(A(P(\mathbb{G})) \geq  \lambda_1(A(P(\mathbb{Z}_{2^kq})) \geq \frac{2^{k-1}(q-1)+\sqrt{2^{2k+1}q(q-1)-3\cdot2^{2k-2}(q-1)^2+2^{k+2}-4}}{2}.
\end{equation}

The proof now follows from (\ref{22}) and (\ref{23}).  \hfill$\Box$

\begin{thm}\label{5.4}
Let $\delta_1$ denote the largest eigenvalue of $Q(P(\mathcal{G}))$. Then, for $k\geq2$ and odd prime $q$, we have
$$\frac{2^{k}(2q-1)+n_2}{2} \leq \delta_1 \leq 2^{k-1}q+2+\sqrt{2^{2k-2}q^2+2^{k}q}+\sqrt{4+2^{k-1}q}+\frac{3\cdot2^{k}q-6+n_3}{2},$$
where $n_2= \sqrt{2^{2k}(4q^2-4q+1)+2^{k+2}(q-1)+8}$ and $n_3=\sqrt{(2^{kq)^2}-2^{k+2}+12}$.
\end{thm}

\noindent{\textbf{Proof.}} To construct the signless Laplacian matrix, we follow the same indexing as in Theorem \ref{5.3}. This gives,
\begin{equation}\label{24}
    Q(P(\mathcal{G}))=\begin{bmatrix}
        Q(P(\mathbb{Z}_{2^{k}q})+V & D_{2^{k}q}\\
        D^t_{2^{k}q} & E_{2^{k}q}
    \end{bmatrix} = \begin{bmatrix}
        Q(P(\mathbb{Z}_{2^{k}q}) & O_2^{k}q\\
        O_2^{k}q & O_2^{k}q
    \end{bmatrix}+ \begin{bmatrix}
        V_{2^{k}q} & D_{2^{k}q}\\
        D^t_{2^{k}q} & E_{2^{k}q}
    \end{bmatrix},\end{equation}
    
where $$V = \begin{bmatrix}
    2^{k}q & 0 & \cdots & 0 \\
    0 & 2^{k-1}q & \cdots & 0 \\
    0 & 0 & \cdots & 0 \\
    \vdots & \vdots & \ddots & \vdots\\ 
    0 & 0 & \cdots & 0
\end{bmatrix}, D = \begin{bmatrix}
    1 & \cdots  & 1  & 1 & \cdots & 1\\
    1 & \cdots  & 1  & 0 & \cdots & 0\\
    0 & \cdots  & 0  & 0 & \cdots & 0\\
    \vdots & \ddots & \vdots & \vdots & \ddots & \vdots \\
    0 & \cdots & 0 & 0 & \cdots & 0
\end{bmatrix}, \text{and} \ E=\begin{bmatrix}
    3 & 1 & \cdots & 0 & 0 & 0 & \cdots & 0\\
    1 & 3 & \cdots & 0 & 0 & 0 & \cdots & 0\\
    \vdots & \vdots & \ddots & \vdots & \vdots & \vdots & \ddots & \vdots\\
    0 & 0 & \cdots & 3 & 1 & 0 & \cdots & 0\\
    0 & 0 & \cdots & 1 & 3 & 0 & \cdots & 0\\
    0 & 0 & \cdots & 0 & 0 & 1 & \cdots & 0\\
    \vdots & \vdots & \ddots & \vdots & \vdots & \vdots & \ddots & \vdots\\
    0 & 0 & \cdots & 0 & 0 & 0 & \cdots & 1
\end{bmatrix}.$$

\noindent Now, by Theorem \ref{t2.3} and Equation (\ref{24}), we get 
\begin{equation}\label{25}
    \delta_1 \leq \delta_1\bigg(\begin{bmatrix}
        Q(P(\mathbb{Z}_{2^{k}q})) & O\\
        O & O
    \end{bmatrix}\bigg)+\delta_1\bigg(\begin{bmatrix}
        V & D\\
        D^t & E
    \end{bmatrix}\bigg).
\end{equation}
Let \begin{equation}\label{26}
    \begin{bmatrix}
        V & D\\
        D^t & E
    \end{bmatrix} = \begin{bmatrix}
        V & C_1 \\
        C^t_1 & O 
    \end{bmatrix}+\begin{bmatrix}
        O & C_2 \\
        C^t_2 & D
    \end{bmatrix},
\end{equation}
where $C_1$ and $C_2$ are the same matrices as in Theorem \ref{5.3}.
\noindent Now using Theorem \ref{t2.3}, we get 
\begin{equation}\label{27}
    \delta_1\bigg(\begin{bmatrix}
        V & D\\
        D^t & E
    \end{bmatrix}\bigg) \leq \delta_1\bigg( \begin{bmatrix}
        V & C_1 \\
        C^t_1 & O 
    \end{bmatrix}\bigg)+\delta_1\bigg(\begin{bmatrix}
        O & C_2 \\
        C^t_2 & D
    \end{bmatrix}\bigg).
\end{equation}
The characteristic polynomial of $\begin{bmatrix}
        V & C_1 \\
        C^t_1 & O 
    \end{bmatrix}$ is $ r_1(x) = \begin{vmatrix}
        xI - V & -C_1 \\
        -C^t_1 & xI 
    \end{vmatrix}.$
    
\noindent Applying $R_1 \rightarrow xR_1$ and then $R_1 \rightarrow R_1+R_{2^kq+1}+R_{2^kq+2}+\cdots+R_{2^{k+1}q}$, we get

$$r_1(x) = \frac{1}{x}\begin{blockarray}{cccc|cccccc}
 \begin{block}{|cccc|cccccc|}
    x(x-2^{k}q)-2^{k}q & 0 & \cdots & 0 & 0 & \cdots & 0 & 0 & \cdots & 0\\
    0 & x-2^{k-1}q & \cdots & 0 & 0 & \cdots & 0 & 0 & \cdots & 0 \\
    \vdots & \vdots & \ddots & \vdots & \vdots & \ddots & \vdots & \vdots & \ddots & \vdots\\
    0 & 0 & \cdots & x & 0 & \cdots & 0 & 0 & \cdots & 0 \\ \hline
    -1 & 0 & \cdots & 0 & x & \cdots & 0 & 0 & \cdots & 0\\
     \vdots & \vdots & \ddots & \vdots & \vdots & \ddots & \vdots & \vdots & \ddots & \vdots\\
     -1 & 0 & \cdots & 0 & 0 & \cdots & x & 0 & \cdots & 0\\
     -1 & 0 & \cdots & 0 & 0 & \cdots & 0 & x & \cdots & 0\\
     \vdots & \vdots & \ddots & \vdots & \vdots & \ddots & \vdots & \vdots & \ddots & \vdots\\
     -1 & 0 & \cdots & 0 & 0 & \cdots & 0 & 0 & \cdots & x\\
\end{block}
 \end{blockarray}.$$
\noindent This gives $r_1(x)= x^{2^{k+1}q-3}(x-2^{k-1}q)(x^2-2^{k}qx-2^{k}q)$.

\noindent Thus, the largest eigenvalue of $\begin{bmatrix}
        V & C_1 \\
        C^t_1 & O 
    \end{bmatrix}$ is 
    \begin{equation}\label{28}
        \delta_1\bigg(\begin{bmatrix}
        V & C_1 \\
        C^t_1 & O 
\end{bmatrix}\bigg) = 2^{k-1}q+\sqrt{2^{2k-2}q^2+2^{k}q}.
    \end{equation}

Similarly, the characteristic polynomial of $\begin{bmatrix}
        O & C_2 \\
        C^t_2 & D
    \end{bmatrix}$ is 
    $r_2(x) =\begin{vmatrix}
        xI & -C_2 \\
        -C^t_2 & xI - D
    \end{vmatrix}.$
    
Applying $R_2 \rightarrow (x-4)R_2$ and $R_2 \rightarrow R_2 + R_{2^{k}q+1}+R_{2^{k}q+2}+\cdots+R_{3\cdot2^{k-1}q},$ we get
$$ = \frac{1}{x-4}\begin{blockarray}{cccc|cccccccc}
\begin{block}{|cccc|cccccccc|}
    x & 0 & \cdots & 0 & 0 & 0 & \cdots & 0 & 0 & 0 & \cdots & 0 \\
    0 & b_1 & \cdots & 0 & 0 & 0 & \cdots & 0 & 0 & 0 & \cdots & 0 \\
    \vdots & \vdots & \ddots & \vdots & \vdots & \vdots & \ddots & \vdots & \vdots & \vdots & \ddots & \vdots \\
    0 & 0 & \cdots & x & 0 & 0 & \cdots & 0 & 0 & 0 & \cdots & 0 \\ \hline
    0 & -1 & \cdots & 0 & x-3 & -1 & \cdots & 0 & 0 & 0 & \cdots & 0 \\
    0 & -1 & \cdots & 0 & -1 & x-3 & \cdots & 0 & 0 & 0 & \cdots & 0 \\
     \vdots & \vdots & \ddots & \vdots & \vdots & \vdots & \ddots & \vdots & \vdots & \vdots & \ddots & \vdots \\
    0 & -1 & \cdots & 0 & 0 & 0 & \cdots & x-3 & -1 & 0 & \cdots & 0 \\ 
    0 & -1 & \cdots & 0 & 0 & 0 & \cdots & -1 & x-3 & 0 & \cdots & 0 \\ 
     0 & 0 & \cdots & 0 & 0 & 0 & \cdots & 0 & 0 & x-1 & \cdots & 0 \\
     \vdots & \vdots & \ddots & \vdots & \vdots & \vdots & \ddots & \vdots & \vdots & \vdots & \ddots & \vdots \\
     0 & 0 & \cdots & 0 & 0 & 0 & \cdots & 0 & 0 & 0 & \cdots & x-1 \\
\end{block}
\end{blockarray},$$ where $b_1=x(x-4)-2^{k-1}q$.
$$r_2(x)= x^{2^{k}q-1}(x-1)^{2^{k-1}q}(x-4)^{2^{k-2}q-1}(x-2)^{2^{k-2}q}(x^2-4x-2^{k-1}q).$$

\noindent Thus, the largest eigenvalue of $\begin{bmatrix}
        O & C_2 \\
        C^t_2 & D
    \end{bmatrix}$ is 
    \begin{equation}\label{29}
     \delta_1\bigg(\begin{bmatrix}
        O & C_2 \\
        C^t_2 & D
\end{bmatrix}\bigg) = 2+\sqrt{4+2^{k-1}q}.
    \end{equation}

By \cite[Theorem 1 and 2]{ba(2020)}, the upper and lower bounds for the spectral radius of the singnless Laplacian matrix of the $P(\mathbb{Z}_n)$ are given by $\frac{3n-6 + \sqrt{8l+n^2-4n+4}}{2}$ and $\frac{2l+n-2+\sqrt{n^2+4l^2+4ln-4n+4}}{2}$ respectively. Now by putting $n=2^kq$ in these bounds and using \cite[Theorem 7.8.13]{crs(2011)}, Equations (\ref{25}), (\ref{27}), (\ref{28}), and (\ref{29}), we get the required result.  \hfill$\Box$

\subsection{Comparision among $\lambda_1({A(P(\mathcal{G}))})$, $\delta_1({Q(P(\mathcal{G}))})$ and their bounds}
The following tables compare the largest eigenvalues of the adjacency matrix and the signless Laplacian matrix to validate Theorems \ref{5.3} and \ref{5.4}. 
Let $\lambda_{upp} (A)$ and $\lambda_{low}(A)$ denote the upper and lower bounds for $\lambda_1(A)$ respectively, and $\delta_{upp} (Q)$ and $\delta_{low}(Q)$ denote the upper and lower bounds for $\delta_1(Q)$ respectively. Then, by Theorems \ref{5.3} and \ref{5.4}, we get the following table for $k=2$ and prime $q$, where $3 \leq q \leq 13$.

\begin{table}[H]
\parbox{.45\linewidth}{
\begin{tabular}{|c|c|c|c|c|}
\hline
Prime ($q$) & $\lambda_{low}(A)$ & $\lambda_1 $ & $\lambda_{upp}(A)$\\
\hline
 3 &  8.24& 9.89 & 16.70\\ 
 \hline
 5 &  14.72& 17.47 & 26.58\\
 \hline
 7 &  21.20 & 25.22 & 36.99\\
 \hline
 11 &  34.14 & 41.16 & 54.30 \\
 \hline
 13 &  40.61 & 49.14 & 63.29\\
 \hline
\end{tabular}
\caption{Comparison table of $\lambda_1({A(P(\mathcal{G}))})$.}
}
\hfill
\parbox{.45\linewidth}{

\begin{tabular}{|c|c|c|c|c|}
\hline
 Prime ($q$) & $\delta_{low}(Q)$ & $\delta_1(Q)$ & $\delta_{upp}(Q)$\\
\hline
3 & 20.77 & 25.36 &38.99\\
\hline
5 & 36.94 & 42.02 & 63.63\\ 
\hline
 7 & 52.60 & 58.68 & 88.16\\
 \hline
 11 &84.56 & 91.87 & 137.03\\
 \hline
 13 & 100.55 & 108.39 & 161.43\\
\hline
\end{tabular}
\caption{Comparison table of $\delta_1({Q(P(\mathcal{G}))})$.}
}
\end{table}

\section{Concluding remark} \label{s6}

The main objective of this paper is to characterise the power graph of split metacyclic groups through a joined union graph representation and find their spectral properties. In \cite{mga(2016)}, the authors expressed $P(\mathbb{Z}_n)$ using a joined union graph decomposition and proved that $P(\mathbb{Z}_n)=K_{\phi(n)+1}+\triangle_n[K_{\phi(d_1)},K_{\phi(d_2)},\ldots,K_{\phi(d_t)}].$
Based on this description, they determined the automorphism groups of $\mathbb{Z}_n$ and $D_{2n}$. This naturally motivates the study of the automorphism group of split metacyclic groups.

\bigskip

\noindent{\bf Acknowledgements :} The first author acknowledges the financial support provided by DST INSPIRE, New Delhi, India (Award No. 03/2022/002991).

\bigskip
\noindent{\bf Conflict of Interest:} The authors declare that there are no conflicts of interest related to this work.

\bigskip
\noindent{\bf Data Availability Statement:} No datasets were generated or analysed during the course of this study.

\end{document}